\documentclass[11pt]{amsart}

\usepackage[hmargin={2.1cm,2.1cm},vmargin={2.5cm,2.5cm}]{geometry}

\usepackage{amsmath, amsthm, amssymb}
\usepackage{mathtools}
\usepackage{graphics}
\usepackage{float}
\usepackage{epsfig}
\usepackage{enumerate}
\usepackage{nicefrac}
\usepackage{cancel}
\usepackage{extarrows}
\usepackage{mathbbol}

\DeclareMathAlphabet{\mathpzc}{OT1}{pzc}{m}{it}
\usepackage{mathrsfs}
\usepackage{stmaryrd}
\usepackage{macros}

\usepackage{hyperref}

\usepackage[usenames,dvipsnames]{xcolor}

\usepackage{stackrel}

\begin{document}

\title[DIRK schemes]{Diagonally implicit Runge-Kutta schemes: \\
Discrete energy-balance laws and compactness properties}

\author[A.J.~Salgado]{Abner J. Salgado}
\address[A.J.~Salgado]{Department of Mathematics, University of Tennessee,
Knoxville TN 37996, USA.}
\email{asalgad1@utk.edu}
\urladdr{https://sites.google.com/utk.edu/abnersg/}

\author[I. Tomas]{Ignacio Tomas}
\address[I. Tomas]{Sandia National Laboratories\footnote{Sandia National
Laboratories is a multimission laboratory managed and operated by National
Technology \& Engineering Solutions of Sandia, LLC, a wholly owned subsidiary
of Honeywell International Inc., for the U.S. Department of Energy’s National
Nuclear Security Administration under contract DE-NA0003525. This paper
describes objective technical results and analysis. Any subjective views or
opinions that might be expressed in the paper do not necessarily represent the
views of the U.S. Department of Energy or the United States
Government.}\footnote{SAND2022-7110 O},
Center for Computing Research, Albuquerque NM, 87185-1320,
USA.}
\email[I. Tomas]{itomas@sandia.gov}
\urladdr{http://www.cs.sandia.gov/ccr-itomas}

\begin{abstract}
We study diagonally implicit Runge-Kutta (DIRK) schemes when applied to
abstract evolution problems that fit into the Gelfand-triple framework. We
introduce novel stability notions that are well-suited to this setting and
provide simple, necessary and sufficient, conditions to verify that a DIRK
scheme is stable in our sense and in Bochner-type norms. We use several
popular DIRK schemes in order to illustrate cases that satisfy the required
structural stability properties and cases that do not. In addition, under some
mild structural conditions on the problem we can guarantee compactness of
families of discrete solutions with respect to time discretization.
\end{abstract}

\keywords{Runge Kutta methods; diagonally implicit; parabolic problems;
hyperbolic problems; stability; compactness.}

\subjclass[2010]{65L05, 
65M12,                  
47J35,                  
65L06,                  
34D20,                  
35K20,                  
35L20,                  
58D25,                  
35K90,                  
35L10,                  
35K15,                  
35L15.                  
}

\maketitle

\section{Introduction}\label{sec:Intro}

The purpose of this work is the study of structure preserving time-marching
schemes for a class of evolution problems in Banach spaces; which essentially
are used to describe possibly degenerate parabolic and hyperbolic initial
boundary value problems. At this stage, it is enough for our purposes to say
that, given a Banach space $\mathbb{V}$, a final time $\tf>0$, and a mapping
$\mathcal{F}:(0,\tf) \times \mathbb{V} \to \mathbb{V}^*$, we seek for $u :
[0,\tf]
\to \mathbb{V}$ that solves
\begin{equation}
\label{eq:TheODE}
  \frac{\diff u}{\diff t} = \mathcal{F}(t,u) , \ t \in (0,\tf],  \qquad u(0)  =u_0.
\end{equation}
As usual, $\mathbb{V}^*$ denotes the dual of $\mathbb{V}$, and
$\langle \cdot , \cdot \rangle$ is the duality pairing; see
Appendix~\ref{AppExampleProb} for a, far from exhaustive, list of example
problems of interest. Solutions of problem \eqref{eq:TheODE}
satisfy the following \emph{energy-balance law}:
\begin{align}\label{AbstractBalance}
  \frac{1}{2} |u(\tf) |^2 - \int_{0}^{\tf} \langle\mathcal{F}(t,u(t)), u(t)
\rangle \diff t
  \leq \frac{1}{2} |u(0)|^2,
\end{align}
where the notation is to be specified. Our goal in this work is to find
numerical schemes that mimic this law. For instance, given a time partition $0 =
t_0 < t_1 < \cdots < t_\frakN = \tf$ with $\dt_n = t_{n+1} - t_{n}> 0$ being the
local timestep size, an approximate solution to \eqref{eq:TheODE} could be
computed by the backward Euler scheme. Ignoring solvability issues, we obtain
$\{u_n \approx u(t_n)\}_{n=1}^\frakN$ and it is immediate to
see that
\begin{align}\label{DiscEnergyBalance}
  \frac{1}{2} |u_\frakN |^2 - \sum_{n = 1}^\frakN \dt_n
\langle\mathcal{F}(t_n,u_n), u_n \rangle \leq \frac{1}{2}|u_0 |^2,
\end{align}
where, once again,  the notation is to be specified.

There are several reasons why discrete energy laws such as
\eqref{DiscEnergyBalance} are important. From the practical point of view, these
are a non-perturbative form of stability: they do not rely on any smallness
assumption, linearization, asymptotic argument, or proximity to some
equilibrium state. For complex PDE problems, where a thorough quantitative and
qualitative analysis of the solution, and corresponding numerical scheme, are
far out of reach, satisfaction of a discrete energy-balance is usually an
excellent surrogate to stability.

On the other hand, from a theoretical point of view, it is often the case that
with the help of discrete energy-balance laws one can assert compactness of
families of numerical solutions. Once again, for complex PDE problems, (weak)
convergence of discrete solutions via compactness is all one can hope to achieve
without introducing additional assumptions. Examples of a successful application
of this approach are numerous. The interested reader is referred to the following
references: elliptic problems \cite{Gidas1981,Evans1987}, hyperbolic
\cite{Tartar1979, Necas1996}, and parabolic \cite{Lions1969}; where one of the
main tools is the Aubin-Lions compactness lemma \cite{Lions1969}, or
some refinement of it \cite{Simon1987, Jungel2012, Jungel2014, Andrei2011,
Gllou2012}.

The backward Euler scheme is one of the simplest time-stepping schemes, yet it
is often possible to prove that it possesses suitable energy-balance laws; see
\eqref{DiscEnergyBalance}. Its major drawback, however, is that it is only first
order accurate. In order to remedy this issue, higher order generalizations,
like Galerkin-in-time schemes, have been developed and analyzed; see for
instance \cite{Lubich1995,Makri2006,Akri2011,Walk2014,Hochbruck2018,Walk2010,
Ern2019} and references therein. These schemes can also be shown to possess
energy-balance laws, and are of arbitrary high order. However, their practical
impact beyond academic examples has been rather limited; see
\cite{Pazner2017, Crockatt2019} for a few exceptions. This is due to the
fact that Galerkin-in-time methods are algebraically equivalent to full-tableau
Runge-Kutta (RK) methods \cite{Akri2011}. As such, they require the solution of
linear algebraic systems where the system matrix is of size $s N \times s N$
for each Newton iteration: here $s$ is the number of stages of the RK method
and $N$ is the number of degrees of freedom of the spatial discretization, see
for instance \cite{Smears2017}.

Diagonally implicit Runge-Kutta (DIRK) methods lie between these two extreme
possibilities (backward Euler and Galerkin-in-time discretizations). DIRK methods
offer significantly added computational benefit, since they only require the
solution of linear systems of size $N \times N$ at each stage, as well as higher
order accuracy, see \cite{Crouzeix1976,Alexander1977, Carpenter2016}. Their
popularity is, perhaps, in big part due to the paper \cite{Ascher1997} which has
been extremely influential in the scientific computing community. The
rigorous study of the mathematical properties of DIRK schemes, however, seems to
be rather underdeveloped. This is particularly the case if we are interested in
the numerical approximation of PDEs satisfying a discrete energy-balance law
such as \eqref{AbstractBalance}. This brings us to the main motivation for
our current work: we wish to present classes of DIRK schemes for which we can
prove discrete energy-balance laws, and study under which conditions the
solution to these schemes enjoy suitable compactness properties.

To achieve these goals we organize our presentation as follows. Notation and
the functional framework we shall operate under are introduced in
Section~\ref{sec:Prelims}. Here we also discuss the minimal set of assumptions
we require on the mapping $\scrF$. Section~\ref{sec:DIRK} then begins by
introducing the general form of DIRK schemes and making some comments regarding
their implementation when applied to PDEs. Two notions of balance laws: discrete
energy-balance, and dissipative discrete energy-balance, respectively, are
introduced here for DIRK schemes; their importance and meaning is also
discussed. This section then presents an exhaustive
literature review regarding the existing (algebraic) notions of stability and
why we believe these are not suitable for our purposes. A list of some popular
two- and three-stage DIRK schemes finalizes this section. The core of our work
is Sections~\ref{sec:TwoStage} and \ref{sec:ThreeStage} where, for two- and
three-stage schemes, respectively, we explore the existence of discrete
energy-balance laws. We arrive at a property, which we call \emph{remarkable
stability} which immediately implies the existence of a dissipative
discrete energy law for a DIRK scheme. We note that, given the Butcher tableau
of an $s$-stage DIRK scheme, verifying remarkable stability only
requires some algebraic manipulations using the entries of the tableau, and the
solution of an algebraic eigenvalue problem of size $s$. Each section concludes
by verifying remarkable stability for a list of widely used DIRK schemes.
Finally, in Section~\ref{sec:AisCoercive}, we refine the assumptions on $\scrF$
to arrive at our strongest notion: discrete Bochner-stability. We show that
every remarkably stable scheme is discretely Bochner-stable, and that (families
of) solutions to discretely Bochner-stable schemes possess suitable compactness
properties.

Finally, we believe that one important feature of our exposition is its
simplicity. The main results of this manuscript rely, in essence, on the simple
polarization identity
\begin{equation}\label{FundIdentity}
  a(a-b) = \tfrac{1}{2}|a|^2 - \tfrac{1}{2}|b|^2 + \tfrac{1}{2} |a-b|^2,
\end{equation}
and some algebraic manipulations. At times these manipulations may be long and
tedious, and some assistance from a computer algebra software
package\footnote{In this manuscript we used Mathematica\copyright.} may be
desired. Despite of this, these are nothing but algebraic manipulations. When
describing compactness properties, well-known elementary order conditions of RK
schemes may be necessary, which are summarized in Appendix~\ref{AppOrderCond}.
No high level tools or specialized notions of algebraic stability for ODE
solvers are used in our work.

\section{Preliminaries}\label{sec:Prelims}

Here we describe the framework and assumptions we shall operate under.
If $p \in (1,\infty)$ its H\"older conjugate is $p'=p/(p-1)$. We extend this
notation such that $1'=\infty$, $\infty'=1$, and $1/p+1/p'=1$ for all $p \in
[1,\infty]$. By $\tf >0$ we denote a final time.

By $A \lesssim B$ we shall mean $A \leq c B$ for a nonessential constant $c$
that may change at each occurrence.

\subsection{Functional framework}\label{FunctAnal}

Throughout our work we shall assume that $\mathbb{H}$ is a separable Hilbert
space with inner product $(\cdot,\cdot)$ and norm $|\cdot|$. By $\mathbb{V}$
we denote a Banach space, and its norm is denoted by $\| \cdot \|$. The dual of
$\mathbb{V}$ is denoted by $\mathbb{V}^*$, the duality pairing between them is
denoted by $\langle \cdot, \cdot\rangle$. The norm in $\mathbb{V}^*$ is denoted
by $\| \cdot \|_*$.

We shall assume that $\mathbb{V} \subset \mathbb{H} \subset \mathbb{V}^*$
is a Gelfand triple. We recall that in this setting the duality pairing is an
extension of the inner-product. In other words, every $v \in \mathbb{H}$
defines an element of $\mathbb{V}^*$ whose action is defined by
\begin{equation}
\label{eq:NeedForGelfand}
  \langle v, w \rangle = (v,w), \quad \forall w \in \mathbb{V}.
\end{equation}
This identification will be repeatedly used in our discussion.

Let $\mathbb{W}$ be a separable Banach space with norm $\|\cdot\|_{\mathbb{W}}$
and $p \in [1,\infty]$. For $w:[0,\tf]\rightarrow \mathbb{W}$ measurable, we
define
\[
  \|w \|_{L^p(0,\tf;\mathbb{W})} = \begin{dcases}
                                     \left( \int_{0}^{\tf}
\|w(s)\|_{\mathbb{W}}^p \diff s\right)^{1/p}, & p<\infty, \\
                                      \esssup_{ s \in [0,\tf]}
\|w(s)\|_{\mathbb{W}}, & p = \infty.
                                   \end{dcases}
\]
Then, we define the Bochner space
\[
  L^p(0,\tf;\mathbb{W}) = \left\{ w:[0,\tf]\rightarrow \mathbb{W} \ \middle|
\|w(s)\|_{L^p(0,\tf;\mathbb{W})}< \infty \right\},
\]
which is Banach for the norm $\| \cdot \|_{L^p(0,\tf;\mathbb{W})}$. The space
of functions $w:[0,\tf] \to \mathbb{W}$ that are continuous is denoted by
$\mathcal{C}([0,\tf];\mathbb{W})$. We endow this space with the
$L^\infty(0,\tf;\mathbb{W})$-norm.

\subsection{Initial value problems}

With all the previous notation and preparations at hand we can proceed to
rigorously describe the class of problems we are interested in approximating. We
assume that the initial condition satisfies $u_0 \in \mathbb{H}$, and that the
slope function satisfies $\mathcal{F}:[0,\tf] \times \mathbb{V} \to
\mathbb{V}^*$. We then seek for $u:[0,\tf] \to \mathbb{H}$ such that
$\tfrac{\diff u}{\diff t} :[0,\tf] \to \mathbb{V}^*$ and
solves \eqref{eq:TheODE}.

We will also assume that $\mathcal{F}$ can be split into an autonomous and
purely non-autonomous time-dependent part as
\[
  \mathcal{F}(t,w) = f(t) - \mathcal{A}(w), \qquad \forall t \in [0,\tf], \quad
\forall w \in \mathbb{V},
\]
where $f:[0,\tf] \to \polV^*$ and $\mathcal{A}: \mathbb{V} \to \mathbb{V}^*$.
In the context of PDEs and/or ODEs on graphs, we assume that nonhomogeneous
boundary data can always be assimilated into $f$.

The minimal set of assumptions we shall impose on the mapping $\mathcal A$ are
as follows.
\begin{enumerate}[$\bullet$]
  \item \emph{Nonnegativity:} The mapping $\mathcal A$ is nonnegative, i.e.,
  \begin{equation}
  \label{positivity}
    \langle \mathcal{A}(w), w \rangle \geq 0, \quad \forall w \in \mathbb{V}.
  \end{equation}

  \item \emph{Local-solvability.} We assume that, for every
$F \in \mathbb{V}^*$, there exists $\beta > 0$ such that for every $\gamma \in
(0,\beta]$ the problem
  \begin{equation}
  \label{eq:AGenericStage}
    v + \gamma \mathcal{A}(v) = F,
  \end{equation}
  has a unique solution $v \in \mathbb{V}$. This assumption guarantees that
nonlinear problems associated to each stage in DIRK schemes have a unique
solution, possibly under some timestep size constraint.
\end{enumerate}

As we shall see below, property \eqref{positivity} is sufficient in order to
derive discrete energy-balance laws for DIRK schemes. While a discrete energy
balance law is not, in general, enough to prove a priori bounds in Bochner-type
norms for the discrete solution or its time derivative, \eqref{positivity}
covers a large family of relevant problems; see Appendix~\ref{AppExampleProb}.
Stronger assumptions on $\calA$ will be imposed in
Section~\ref{sec:AisCoercive}, and these will allow us to establish a priori
bounds in Bochner-type norms.

\section{DIRK schemes}\label{sec:DIRK}

In this section we recall some general notions related to RK schemes; and,
in particular, present some details regarding DIRK schemes. We will also detail
the main stability notion for these schemes that we shall pursue.

We recall that RK schemes are uniquely characterized by their so-called Butcher
tableau
\begin{equation}
\label{eq:TheDIRK}
  \begin{array}{c|c}
    \bv{c} & \bv{A} \\
    \hline
      & \bv{b}^\intercal
  \end{array}\ .
\end{equation}
Here $s \in \mathbb{N}$ is the number of stages,
$\bv{A} = [a_{ij}] \in \mathbb{R}^{s \times s}$ is the matrix of coefficients,
$\bv{c} =[c_i]\in [0,1]^s$ are the pseudo-collocation times, and $\bv{b} =[b_i]
\in [0,1]^s$ are the weights. For the sake of completeness, necessary order
conditions for \eqref{eq:TheDIRK} are summarized in
Appendix~\ref{AppOrderCond}. We remind the reader that a DIRK scheme is one
where the matrix $\bv{A}$ is lower triangular.

Let us now, to make things precise, detail how a DIRK scheme is applied to
\eqref{eq:TheODE}. Let $\frakN \in \polN$ be the number of steps. We introduce a
partition $\calP_\frakN = \{t_n\}_{n=0}^\frakN$ of $[0,\tf]$, i.e.,
\[
  0 = t_0  < \cdots < t_\frakN = \tf,
\]
and set $\dt_n = t_{n+1} - t_n$, for $n=0, \ldots, \frakN-1$ to be the local
timestep. Starting from $u_0$ we will compute the sequence $\{u_n\}_{n=1}^\frakN
\subset \polV$ such that $u_n \approx u(t_n)$ as follows. For $n \geq 0$ we
solve the following equations in $\polV^*$
\begin{equation}
\label{eq:DIRKappliedtoPDE}
  \begin{dcases}
    v_{n,i} = u_n + \dt_n \sum_{j=1}^i a_{ij} \scrF_{n,j}, &i = 1, \ldots, s, \\
    u_{n+1} = u_n + \dt_n \sum_{j=1}^s b_j \scrF_{n,j},
  \end{dcases}
\end{equation}
where we introduced a shorthand notation $\scrF_{n,j}$ defined by
\[
  f_{n,j} = f(t_n + c_j \dt_n ), \qquad
  \scrF_{n,j} = f_{n,j} - \scrA(v_{n,j}).
\]
For fixed $n \in \{1, \ldots, \frakN\}$, the quantities $\{v_{n,i}\}_{i=1}^s$
are called the stages.

Notice that a generic step in \eqref{eq:DIRKappliedtoPDE} requires, given
$F \in \polV^*$, to find $v \in \polV$ that solves \eqref{eq:AGenericStage}
with, for some $i \in \{1, \ldots, s\}$, $\gamma= a_{ii}\dt_n$. Thus, owing to
the local solvability condition, for this scheme to be well-defined for any
partition $\calP_\frakN$, it is sufficient to require that $a_{ii}>0$. For this
reason, in what follows we shall assume that this is always the case for a DIRK
scheme.

Finally, observe that, although the equations are posed in $\polV^*$, it is not
difficult to show that, for all $n\geq 0$, $\{v_{n,i}\}_{i=1}^s \cup \{u_{n+1}
\} \subset \polV$.

\begin{remark}[smoothness of the right hand side]
We comment that usually the theory regarding well-posedness of
\eqref{eq:TheODE}, only requires the right hand side $f$ to be such that $f \in
L^{r}(0,\tf;\polV^*)$ for some $r>1$. This makes, in general, the quantities
$f_{n,j}$, for $n\geq0$ and $j = 1, \ldots, s$, meaningless, as point
evaluations of $f$ are not possible. This usually is circumvented by replacing
$f$ by a suitable approximation $f_{\calP_\frakN} \in
\mathcal{C}([0,\tf];\polV^*)$. In order to avoid unnecessary clutter of the
notation we will ignore this issue.
\end{remark}

\begin{remark}[interpretation]\label{rem:interpretationWeak}
We note that \eqref{eq:DIRKappliedtoPDE} must be understood by its action
against suitable ``test functions''. Owing to \eqref{eq:NeedForGelfand}, a
generic stage $v_{n,i} \in \polV$ for $i =1, \ldots, s$ must be such that, for
every $w \in\polV$,
\[
  ( v_{n,i} , w ) + \tau_n a_{ii} \langle \calA(v_{n,i}), w \rangle = (
u_n, w ) + \tau_n\sum_{j=1}^i a_{ij} \langle f(t_n + c_j \tau_n), w
\rangle - \tau_n\sum_{j=1}^{i-1} a_{ij} \langle \calA( v_{n,j}), w
\rangle.
\]
\end{remark}

\subsection{Discrete energy-balance} \label{sub:discenergy}

As we described above, $s$-stage DIRK methods are characterized by its tableau,
and are such that the matrix $\bv{A}$ is lower triangular and has positive
diagonal entries. The following definition will be our main notion of stability
for DIRK schemes.

\begin{definition}[discrete energy-balance]\label{DiscBalanceDef}
We say that a DIRK scheme with tableau \eqref{eq:TheDIRK}, where $\bv{A}$ is
lower triangular and has positive diagonal entries, satisfies a discrete
energy-balance if, for any $\frakN \in \polN$, every partition $\calP_\frakN$,
and all $n \in \{0, \ldots, \frakN-1\}$, we have that
\begin{equation}\label{AbstractDiscLaw}
  \frac12 |u_{n+1}|^2 + \sum_{i=1}^{s+1} \delta_i |v_{n,i} - v_{n,i-1}|^2 -
\dt_n \sum_{i=1}^s \nu_{ii}\langle \scrF_{n,i}, v_{n,i} \rangle =
  \frac12 |u_n|^2 + \dt_n \sum_{i=1}^s \sum_{j=i+1}^s \nu_{ij} \langle
\scrF_{n,i}, v_{n,j} \rangle,
\end{equation}
where we introduced the notation $v_{n,0} = u_n$ and $v_{n,s+1} = u_{n+1}$.
The coefficients $\{\delta_i\}_{i=1}^{s+1} \subset \mathbb{R}$ and
$\{\nu_{ij}\}_{i=1,j=i}^{s,s}\subset \mathbb{R}$ depend only on the tableau,
and are expected to satisfy the following constraints:
\begin{enumerate}[$\bullet$]
  \item $\delta_i \geq 0$ for all $i \in \{1,\ldots,s+1\}$.

  \item $\nu_{ii}>0$ for all $i=\{1, \ldots, s\}$.

  \item The coefficients $\{\nu_{ij}\}_{i=1,j=i}^{s,s}$ must satisfy the
constraint
  \begin{equation}\label{QuadCond}
    \sum_{i = 1}^s \nu_{ii} + \sum_{i=1}^s \sum_{j = i+1}^s \nu_{ij} = 1.
  \end{equation}
\end{enumerate}
\end{definition}

Notice that \eqref{AbstractDiscLaw} differs from \eqref{DiscEnergyBalance} in at
least two salient terms. First, the presence of the terms $\sum_{i=1}^{s+1}
\delta_i |v_{n,i} - v_{n,i-1}|^2$ may seem out of place. However, these terms
represent purely numerical artificial damping and are usually quite desirable
when dealing with dissipative/parabolic PDEs\footnote{For instance, in some very
specific contexts such as projection methods for incompressible Navier-Stokes
equations, artificial damping terms are critical to guarantee numerical
stability of the scheme, see for instance \cite{GS2009,GS2011} and references
therein.}. However, artificial damping is not desirable when considering the
discretization of Hamiltonian, or related, PDEs (e.g., PDEs that preserve
quadratic invariants). For this reason, we allow the coefficients $\delta_i$ to
be all zero if that is suitable for the problem at hand. Second, while the terms
$\langle \scrF_{n,i}, v_{n,i} \rangle$ are expected, the off-diagonal terms,
i.e., $\langle \scrF_{n,i}, v_{n,j} \rangle$ with $i \not = j$, may be
problematic and, thus, it is possible that \eqref{AbstractDiscLaw} will not
lead to suitable a priori estimates for discrete solutions. For this reason, we
introduce a stronger notion, which we call ``dissipative discrete
energy-balance''.

\begin{definition}[dissipative discrete
energy-balance]\label{def:DiscrDissEnergyBalance}
We say that a DIRK scheme with tableau \eqref{eq:TheDIRK} satisfies a
dissipative discrete energy-balance if there are strictly positive
$\{\nu_i\}_{i=1}^s$, and 
\[
  \mathcal{Q}: [\polH]^{s+1} \to \mathbb{R}, \qquad
\mathcal{Q}(w_1,\ldots,w_{s+1}) = \frac12 \sum_{i=1}^{s+1} \sum_{j=1}^{s+1}
q_{ij} (w_i,w_j), \qquad q_{ij}=q_{ji} \in \mathbb{R},
\]
a nonnegative definite quadratic form on $[\polH]^{s+1}$ such that, for any
$\frakN \in \polN$, every partition $\calP_\frakN$, and all $n \in \{0, \ldots,
\frakN-1\}$, we have
\begin{equation}
\label{AbstractDiscBal}
  \frac{1}{2}|u_{n+1}|^2 + \mathcal{Q}(u_n, v_{n,1}, \ldots, v_{n,s} ) - \dt_n
\sum_{i = 1}^s \nu_{i} \langle \scrF_{n,i},v_{n,i} \rangle = \frac{1}{2}
|u_{n}|^2 .
\end{equation}
\end{definition}

Expression \eqref{AbstractDiscBal} states that, beyond the ``energy''
introduced into the system by the non-autonomous part $f$, the scheme will not
produce any spurious surplus of energy. Moreover, the fact that $\mathcal{Q}$ is
nonnegative definite implies that the scheme may even dissipate some energy. The
following result makes this intuition rigorous.

\begin{proposition}[discrete energy
dissipation]\label{prop:DiscEnergyDissipationMildAssume}
Assume that, in \eqref{eq:TheODE}, the mapping $\calA$ satisfies
\eqref{positivity}. If the solution to \eqref{eq:TheODE} is approximated with a
DIRK scheme that satisfies the dissipative discrete energy-balance of
Definition~\ref{def:DiscrDissEnergyBalance}, then for all $\frakN \in \polN$,
any partition $\calP_\frakN$, and all $n \in \{0, \ldots, \frakN-1\}$, we have
\begin{align}\label{IncEstimate}
 \frac12 |u_{n+1} |^2 \leq \frac12 |u_n|^2 + \dt_n \sum_{i=1}^s \nu_i \langle
f_{n,i}, v_{n,i} \rangle
\end{align}
\end{proposition}
\begin{proof}
We start by recalling that $\scrF_{n,i} = f_{n,i} - \calA(v_{n,i})$.
Since $\mathcal Q$ and the diagonal terms $\langle \mathcal{A}(v_{n,i}),
v_{n,i}\rangle$ are nonnegative due to the assumption \eqref{positivity}, they
can be dropped from the left hand side of \eqref{AbstractDiscBal}.
\end{proof}

\subsection{Literature review}\label{sub:Literature}

At this point, it is worth making some comments about RK methods, discrete
energy laws and the preexisting literature.

There is a significant body of numerical ODE literature attempting to bridge
the gap between algebraic notions of stability and nonlinear notions of
stability; see for instance: A-stability \cite{Dahlquist1963}, L-Stability
\cite{Ehle1969,Ehle1973}, B-Stability \cite{Butcher1975}, AN-Stability
\cite{Burrage1979}, BN-stability \cite{Burrage1979}, BS-stability and
BSI-stability \cite{Frank1985I,Frank1985II} and G-stability \cite{Dahlquist1978}
among many others (see also \cite{Hairer1993I,Butcher2008}). Similarly, there is
a specific body of scientific literature relating algebraic notions of stability
and discrete energy laws
\cite{Burrage1979,Burrage1980,Humph1994,Ranocha2020,Wu2022}. However, we find
that the classical techniques used in the numerical ODE literature are largely
incompatible with our current goals, and the notions of stability that we are
trying to advance; see
Definitions~\ref{DiscBalanceDef}--\ref{def:DiscrDissEnergyBalance} and Definition~\ref{def:discBochnerStable} below. We explain our reasoning in more detail.

\begin{enumerate}[$\bullet$]

  \item \emph{Functional setting.} A common assumption in the numerical ODE
literature is that $\mathcal{A}:\mathbb{H} \rightarrow \mathbb{H}$ boundedly;
see, for instance \cite{Ranocha2018,Ranocha2020,Higueras2005,Wu2022}. This is a
rather stringent assumption that, in general, is not suitable for PDEs. For
instance, it does not allow us to capture the linear heat equation,
incompressible Navier-Stokes equation, nor large families of
advection-reaction-diffusion systems. For this reason we, instead, focus on the
Gelfand-triple functional framework and assume that $\mathcal{A}:\mathbb{V}
\rightarrow \mathbb{V}^*$.

  \item \emph{Choice of norms.} The numerical ODE literature focuses on the
development of $L^\infty(0,\tf;\polH)$ estimates, i.e., $|u_n| \leq |u_0|$ for
all $n\geq0$; see, for instance \cite{Burrage1980,Humph1994}. For ODEs in
$\mathbb{R}^N$ and some limited cases of linear hyperbolic PDEs, this may
suffice. However, for many PDE problems, such as parabolic-like problems, this
type of estimate may be insufficient. Without a priori bounds on spatial
derivatives in Bochner-type norms, it is not possible to assert stability of
such schemes. To assert convergence, usually, one additionally requires an a
priori estimate on the time derivative, again in a Bochner-type norm. In
other words: discrete energy balances of the form \eqref{AbstractDiscBal} are
suitable for the analysis of parabolic-like problems, while estimates of the
form \eqref{IncEstimate}, in general cannot yield enough compactness.

We note that estimates in $L^\infty(0,\tf;\polH) \cap L^p(0,\tf;\polV)$, for
some $p>1$, are standard in the PDE and numerical-PDE literature
\cite{Lions1969, MarionTemam1998, Thomee2006, Walk2010, Gllou2012, Feng2006,
Layton2008, Shen2010, Roubicek2013, Elliott2021}. On the other hand, to our
knowledge, the numerical ODE literature \cite{Hairer1993I, HumphBook,
Butcher2008, Hairer1996II} has not focused on a priori bounds in Bochner-type
norms, space-time compactness, or convergence without regularity assumptions for
problems of growing dimensionality (i.e. discretization of evolutionary PDEs).

 \item \emph{Finite dimensionality.} We want to develop stability results that
are valid for finite dimensional problems as well as their infinite dimensional
limits. This is a somewhat delicate issue when dealing with operators of the
form $\mathcal{A}:\mathbb{V} \rightarrow \mathbb{V}^*$. Let us explain what we
mean here. As we detailed above, see \eqref{eq:AGenericStage} and
Remark~\ref{rem:interpretationWeak}, a generic stage must be interpreted as:
find $v \in \polV$ such that
  \[
    (v,w) + a \dt_n \langle \calA(v), w\rangle =\langle F,w\rangle, \quad
\forall w\in \polV.
  \]
  At this point one may be tempted to set $w = \calA(v)$. However, that is not
necessarily well-defined unless additional assumptions are made\footnote{For
instance, if $\mathcal{A} = - \Delta$, the product
$(\mathcal{A}(v),\mathcal{A}(v) )_{L^2(\Omega)}$ is meaningful only if we assume
$H^2(\Omega)$-regularity of $v$.}. Similarly, higher order compositions of the
operator, i.e., $\mathcal{A}(\mathcal{A}(v))$, are not meaningful unless
$\mathcal{A}$ maps a Banach space to itself. We note that energy identities and
a priori bounds in norm using such constructions are common in the numerical ODE
literature; see for instance \cite{Ranocha2018,Ranocha2020,Higueras2005, Wu2022}
and the review paper \cite[p. 1464--1465]{Humph1994}.
\end{enumerate}

In light of the shortcomings mentioned above, in this manuscript we develop
stability results that:

\begin{enumerate}[$\bullet$]
  \item Target specifically DIRK schemes in the framework of a Gelfand triple
and unbounded operators. We limit ourselves to the case of two- and three-stage
schemes.

  \item Energy identities and a priori bounds will solely rely on the inner
product in $\mathbb{H}$, the duality pairing $\langle \mathcal{A}(u),v\rangle$,
and additive telescopic cancellation arguments. We do not use or invoke higher
order compositions of the operator $\mathcal{A}$, higher order products such as
$\langle\mathcal{A}(u),\mathcal{A}(v)\rangle$, nor similar ``multiplicative''
constructions.

  \item While this may be necessary to show existence of solutions to
\eqref{eq:TheODE}, we make no assumption of contractivity/monotonicity of our
operator $\mathcal{A}$ to obtain stability. Our primary notion of nonlinear
stability revolves around ``dissipative discrete energy-balances'', see
\eqref{AbstractDiscBal}, which is, strictly speaking, a property of the scheme,
not a property of the operator. This enables the proof of a priori bounds in
Bochner-type norms for $u$ and its time derivative $\tfrac{\diff u}{\diff t}$
for some families of operators; see Proposition~\ref{prop:DiscBochnerIsImplied}.

  \item Our a priori Bochner-type norm estimates make a clear cut distinction
between ``artificial/numerical damping'' and ``physical or PDE dissipation''
terms. A precise identification of artificial damping terms plays a pivotal role
in order to establish stability of the scheme. On the other hand, physical
dissipation is fundamental to establish dual norm estimates on the time
derivative of the discrete solution.
\end{enumerate}

As a final comment we note that the mathematical theory about Galerkin-in-time
and/or full-tableau RK methods developed by the numerical PDE community; see
\cite{Lubich1995, Makri2006, Akri2011, Walk2014, Hochbruck2018, Ern2019} and
references therein, rarely ever applies (or even mentions) DIRK schemes. We do
however, highlight the preexistence of quite relevant material with specific
focus on DIRK schemes that shares a few common points of intersection with the
material presented in this manuscript. In particular \cite{Shin2020} addresses
the issue of gradient flow stability for DIRK schemes, while
\cite{Osterman2010} addresses convergence without regularity assumptions using
semi-group methods. However, the material presented in the current paper
differs quite significantly in relationship to the mathematical tools used and
the degree of generality.

\subsection{Some popular DIRK schemes}\label{sub:ExamplesOfDIRK}

Let us present some popular two- and three-stage DIRK schemes, and briefly
mention some of their known properties. These will be our specific examples
under consideration used to illustrate the developed theory.

\subsubsection{Two-stage schemes}\label{subsub:TwoStageList}
We will consider the following two-stage schemes.
\begin{enumerate}[$\bullet$]
  \item Alexander's DIRK22 scheme:
  \begin{align}\label{AlexDIRK22tableu}
    \begin{array}
    {c|cc}
    \gamma & \gamma                 \\
    1      & 1-\gamma & \gamma \\
    \hline
           & 1-\gamma & \gamma
    \end{array}
  \qquad \qquad \gamma = 1 - \tfrac{\sqrt{2}}{2}.
  \end{align}
  Tableau \eqref{AlexDIRK22tableu} seems to appear for the first time in
\cite{Alexander1977}.

  \item Butcher-Burrage DIRK22 scheme:
  \begin{align}\label{ButcherBurrageTableau}
  \begin{array}
  {c|cc}
  \gamma     & \gamma                 \\
  1 - \gamma & 1-2\gamma     & \gamma \\
  \hline
             & \tfrac{1}{2} & \tfrac{1}{2}
  \end{array}
  \qquad \qquad \gamma = 1 \pm \tfrac{\sqrt{2}}{2}.
  \end{align}
  To the best of our knowledge this tableau appears for the first time in
\cite[p. 51]{Burrage1979}.

  \item  Kraaijevanger-Spijker DIRK22 scheme
  \begin{align}\label{KraaijDIRK22}
  \begin{array}
  {c|cc}
  \tfrac{1}{2} & \tfrac{1}{2}                 \\
  \tfrac{3}{2} & -\tfrac{1}{2} & 2 \\
  \hline
              & - \tfrac{1}{2} & \tfrac{3}{2}
  \end{array}
  \end{align}
  see \cite[p. 77]{Kraaije1989}.

  \item  Crouzeix's DIRK23 scheme:
  \begin{align}\label{CrouzeixTableau}
  \begin{array}
  {c|cc}
  \tfrac{1}{2}+ \gamma & \tfrac{1}{2}+ \gamma &                     \\
  \tfrac{1}{2}- \gamma & -2 \gamma            & \tfrac{1}{2}+ \gamma \\
  \hline
         & \tfrac{1}{2} & \tfrac{1}{2}
  \end{array}
  \qquad \qquad \gamma = \tfrac{\sqrt{3}}{6}.
  \end{align}
  This tableau can be found in \cite{Crou1979}. This scheme is third order
accurate.
\end{enumerate}


\subsubsection{Three-stage schemes}\label{subsub:ThreeStageList}
Regarding three-stage methods we will consider:
\begin{enumerate}[$\bullet$]

  \item Alexander's DIRK33 scheme
  \begin{align}\label{DIRK33Lstable}
  \begin{array}
  {c|ccc}
  \gamma                & \gamma                &             & \\
  \tfrac{1 + \gamma}{2} & \tfrac{1 - \gamma}{2} & \gamma      & \\
  1                     & b_1(\gamma)           & b_2(\gamma) & \gamma \\
  \hline
                        & b_1(\gamma)           & b_2(\gamma) & \gamma
  \end{array}
  \end{align}
  where $\gamma$ is the root of $6\gamma^3 - 18 \gamma^2 + 9\gamma -1 = 0$ in
the interval $(\tfrac16,\tfrac12)$. More precisely we have that
  \begin{align*}
    \gamma &= 1 - \sqrt{2} \sin \big(\tfrac{2 \arctan
\frac{\sqrt{2}}{2}}{3}\big) \approx 0.4358665 \\
    b_1(\gamma) &= -\tfrac{3}{2} \gamma^2 + 4 \gamma - \tfrac{1}{4} \approx
1.2084966 \\
    b_2(\gamma) &= \tfrac{3}{2} \gamma^2 - 5 \gamma + \tfrac{5}{4} \approx
-0.6443631.
  \end{align*}
  Tableau \eqref{DIRK33Lstable} seems appear for the first time in
\cite[p. 1012]{Alexander1977}.

  \item  N{\o}rsett DIRK34 order method:
  \begin{align}\label{CrouRavTableau}
  \begin{array}
  {c|ccc}
  \gamma       & \gamma        & & \\
  \tfrac{1}{2} & \tfrac{1}{2} - \gamma & \gamma  & \\
  1 - \gamma   & 2 \gamma & 1 - 4 \gamma & \gamma   \\
  \hline
               & \delta
               & 1 - 2 \delta
               & \delta
  \end{array}
  \qquad \qquad \delta = \tfrac{1}{6(1-2\gamma)^2},
  \end{align}
  and $\gamma$ is one of the roots of the equation
  $\gamma^3 - \tfrac{3}{2} \gamma^2 + \frac{1}{2} \gamma - \tfrac{1}{24} = 0$.
  More precisely these roots are
  \begin{align}
  \begin{split}\label{GammasCroRav}
    \gamma_1 &= \tfrac{\sqrt{3}}{3} \cos(\tfrac{\pi}{18}) + \tfrac{1}{2} \approx
1.068579021, \\
    \gamma_2 &= \tfrac{1}{2} - \tfrac{\sqrt{3}}{3} \sin(\tfrac{2\pi}{9}) \approx
0.1288864005,  \\
    \gamma_3 &= \tfrac{1}{2} - \tfrac{\sqrt{3}}{3} \sin(\tfrac{\pi}{9}) \approx
0.3025345781.
  \end{split}
  \end{align}
  The case of $\gamma_1$ appears in the literature as the Crouzeix-Raviart
scheme \cite{Crou1979}.
\end{enumerate}

\subsection{An alternative representation of DIRK
schemes}\label{sub:Extrapolation}

Let us finish the general discussion about DIRK schemes with an alternative
representation of the solution at the next time step as an extrapolation. This
will be useful when deriving discrete energy-balance laws.

\begin{proposition}[extrapolation]\label{prop:extrapolateDIRK}
Assume that the RK scheme with tableau \eqref{eq:TheDIRK} is such that
$\bv{A}$ is invertible. Then $u_{n+1}$, the solution at the next discrete time,
has the following representation
\[
  u_{n+1} = \left( 1- \sum_{i=1}^s \lambda_i \right) u_n + \sum_{i=1}^s
\lambda_i v_{n,i},
\]
where
\[
  \lambda_i = \bv{b}^\intercal \bv{A}^{-1} \bv{e}_i,
\]
and $\{\bv{e}_i\}_{i=1}^s$ is the canonical basis of $\mathbb{R}^s$.
\end{proposition}
\begin{proof}
For $n \geq 0$ we introduce the notation
\[
  \bv{v}_n = \left[ v_{n,1}, \ldots, v_{n,s} \right]^\intercal \in \mathbb{R}^s,
\qquad
  \boldsymbol{\scrF}_n = \left[ \scrF_{n,1}, \ldots, \scrF_{n,s}
\right]^\intercal \in \mathbb{R}^s, \qquad
  \boldsymbol{1} = \left[1, \ldots, 1\right]^\intercal \in \mathbb{R}^s,
\]
Therefore, the stages in \eqref{eq:TheDIRK} can be rewritten as
\[
  \bv{v}_n = u_n \boldsymbol{1} + \dt_n \bv{A} \boldsymbol{\scrF}_n \quad \iff
\quad \boldsymbol{\scrF}_n = \frac1{\dt_n} \bv{A}^{-1} \left( \bv{v}_n - u_n
\boldsymbol{1} \right),
\]
where we used that the matrix $\bv{A}$ is invertible.

We can then write the solution at the next discrete time as
\[
  u_{n+1} = u_n + \dt_n \bv{b}^\intercal \boldsymbol{\scrF}_n = \left( 1 -
\bv{b}^\intercal \bv{A}^{-1} \boldsymbol{1} \right) u_n + \bv{b}^\intercal
\bv{A}^{-1} \bv{v}_n.
\]
Setting
\[
  \bv{b}^\intercal \bv{A}^{-1} = [\lambda_1, \ldots, \lambda_s] \in \mathbb{R}_s
\]
the result follows.
\end{proof}


\section{Discrete energy-balance for two-stage schemes}\label{sec:TwoStage}

Here we study discrete energy-balance laws for two-stage schemes. Our main
contribution here is to find a class of schemes, which we will call
\emph{remarkably stable}, see Definition~\ref{RemarkableDefTwoStage}, which
automatically satisfy a dissipative discrete energy-balance, see
Definition~\ref{def:DiscrDissEnergyBalance}.

\subsection{General discrete energy-balance
laws}\label{sub:AbstractLawsTwoStage}

We begin by specializing Proposition~\ref{prop:extrapolateDIRK} for the case of
two-stage schemes.

\begin{corollary}[two-stage extrapolation]
\label{cor:TwostageExtrapolate}
Let the RK scheme \eqref{eq:TheDIRK} be such that $s=2$, $\bv{A}$ is lower
triangular, and with positive diagonal entries. Then, for all $n\geq 0$,
\begin{equation}
\label{FinalSolLinearDirk22}
  u_{n+1} = \left( 1- \lambda_1 - \lambda_2 \right) u_n + \lambda_1 v_{n,1} +
\lambda_2 v_{n,2},
\end{equation}
with
\begin{equation}
\label{Lambda12general}
  \lambda_1 = \frac{b_1}{a_{11}} - \frac{b_2 a_{21}}{a_{22} a_{11}} ,
  \qquad \qquad
  \lambda_2 = \frac{b_2}{a_{22}} .
\end{equation}
\end{corollary}
\begin{proof}
A direct computation shows that
\[
  \bv{A}^{-1} = \begin{bmatrix}
                  \frac1{a_{11}} & 0 \\ - \frac{a_{21}}{a_{11}a_{22}} &
\frac1{a_{22}}
                \end{bmatrix},
  \quad  \bv{b}^\intercal \bv{A}^{-1} = [b_1 \ b_2] \begin{bmatrix}
                  \frac1{a_{11}} & 0 \\ - \frac{a_{21}}{a_{11}a_{22}} &
\frac1{a_{22}}
                \end{bmatrix} = \begin{bmatrix} \frac{b_1}{a_{11}} - \frac{b_2
a_{21}}{a_{22} a_{11}} \\ \frac{b_2}{a_{22}} \end{bmatrix},
\]
as we intended to show.
\end{proof}

\begin{lemma}[some useful identities]
\label{lem:InnProds}
Let $s=2$, $N \in \polN$, $\calP_\frakN$ be any partition of $[0,\tf]$, and
$n \in \{0, \ldots, \frakN-1\}$. If in \eqref{eq:TheDIRK} the matrix $\bv{A}$ is
lower triangular and with positive diagonal entries, then we have
\begin{align}
\label{eq:id1}
  | v_{n,1} |^2 + | v_{n,1} -u_n |^2 - | u_n |^2 &= 2a_{11}\tau_n \langle
\scrF_{n,1}, v_{n,1} \rangle, \\
\label{eq:id2}
  | v_{n,2} |^2 + | v_{n,2} - v_{n,1} |^2 - | v_{n,1} |^2 &= 2\tau_n \left[
(a_{21}-a_{11}) \langle \scrF_{n,1}, v_{n,2} \rangle + a_{22} \langle
\scrF_{n,2}, v_{n,2} \rangle \right], \\
\label{eq:id3}
  | v_{n,1} |^2 - ( u_n, v_{n,1} ) &= \tau_n a_{11} \langle \scrF_{n,1}, v_{n,1}
\rangle, \\
\label{eq:id4}
  | v_{n,2} |^2 - (u_n,v_{n,2}) &= \tau_n \left[ a_{21} \langle \scrF_{n,1},
v_{n,2} \rangle + a_{22} \langle \scrF_{n,2}, v_{n,2} \rangle \right], \\
\label{eq:id5}
  | v_{n,2} |^2 - (v_{n,1},v_{n,2}) &= \tau_n \left[ (a_{21} - a_{11})\langle
\scrF_{n,1}, v_{n,2} \rangle + a_{22}\langle \scrF_{n,2},v_{n,2} \rangle
\right],
\end{align}
where $\{u_n,v_{n,1},v_{n,2},u_{n+1}\}$ come from \eqref{eq:DIRKappliedtoPDE}.
\end{lemma}
\begin{proof}
These identities follow from taking duality pairings of each of the stages
with suitable functions. Identity \eqref{eq:id1} comes from testing the equation
for the first stage with $2v_{n,1}$ and using the well--known polarization
identity \eqref{FundIdentity}. Similarly, identity \eqref{eq:id2} comes from
testing the second stage with $2v_{n,2}$. Identity \eqref{eq:id3} comes from
testing the first stage with $v_{n,1}$. Similarly, we get \eqref{eq:id4} by
testing the second stage with $v_{n,2}$. Finally, we combine the two stages and
test the result with $v_{n,2}$ to obtain \eqref{eq:id5}.
\end{proof}

We are now in position to prove a precursor to \eqref{AbstractDiscLaw}. Notice
that here we are not assuming any order conditions on the entries of the Butcher
table.

\begin{theorem}[discrete energy identity I]
\label{thm:protoenergy}
Let $s=2$, $N \in \polN$, $\calP_\frakN$ be a partition of $[0,\tf]$, and
$n \in \{0, \ldots, \frakN-1\}$. If $a_{ii}>0$ for all $i = \{1,\ldots,s\}$,
then the solution to \eqref{eq:DIRKappliedtoPDE} satisfies
\begin{multline}\label{eq:protoenergy}
  \frac12 | u_{n+1} |^2 + \delta_1 | v_{n,1} - u_n |^2 + \delta_2 | v_{n,2} -
v_{n,1} |^2
  -\tau_n \left[ \nu_{11} \langle \scrF_{n,1}, v_{n,1} \rangle + \nu_{22}
\langle \scrF_{n,2},v_{n,2} \rangle \right] = \\
  \frac12 | u_n |^2 + \tau_n \nu_{12} \langle \scrF_{n,1},v_{n,2} \rangle,
\end{multline}
where


\begin{equation}
\label{DisipCoeffDirk22}
\begin{gathered}
\delta_1 = \frac12(\lambda_1 + \lambda_2)(2 -\lambda_1 - \lambda_2) \, , \
\delta_2 = \frac12 \lambda_2(2 - \lambda_2) \, , \
\nu_{11} = a_{11} \left[ \lambda_1(1-\lambda_2) +
\lambda_2(2-\lambda_2) \right] \, , \\
\nu_{22} = a_{22} \lambda_2 \, , \
\nu_{12} = \lambda_2 \left[ a_{21} + a_{11}(\lambda_1 + \lambda_2-2)
\right] \, .
\end{gathered}
\end{equation}

and $\lambda_1, \ \lambda_2$ where introduced in
Corollary~\ref{cor:TwostageExtrapolate}.
\end{theorem}
\begin{proof}
We begin by taking the inner product of the extrapolation identity
\eqref{FinalSolLinearDirk22} with itself to obtain
\begin{align*}
  | u_{n+1} |^2 &= \left( 1 - \lambda_1 - \lambda_2 \right)^2 | u_n |^2 +
\lambda_1^2 | v_{n,1} |^2 + \lambda_2^2 | v_{n,2} |^2 \\
  &+2\lambda_1 \left( 1 - \lambda_1 - \lambda_2 \right) (u_n, v_{n,1})
  +2\lambda_2 \left( 1 - \lambda_1 - \lambda_2 \right) (u_n, v_{n,2})
  +2\lambda_1 \lambda_2 (v_{n,1}, v_{n,2}).
\end{align*}

The rest of the proof entails lengthy but trivial computations. One merely
has to substitute \eqref{eq:id1}--\eqref{eq:id5} in the previous identity. The
reader is encouraged to launch their favorite computer algebra system to verify
these computations.
\end{proof}

\begin{remark}[consistency check] A direct computation shows that
$\nu_{11} +\nu_{22} + \nu_{12} = b_1 + b_2$. Assuming that the scheme satisfies
second order conditions, see Appendix \ref{AppOrderCond}, we have that $\nu_{11}
+\nu_{22} + \nu_{12} = 1$.
\end{remark}

We now write a precursor to the dissipative energy identity
\eqref{AbstractDiscBal}.

\begin{corollary}[discrete energy identity II]\label{PropRewritten22}
The discrete energy-balance law \eqref{eq:protoenergy} can be rewritten as
\begin{equation}
\label{GeneralEnergyBal}
\frac{1}{2}|u_{n+1}|^2 + \mathcal{Q}(u_n,v_{n,1},v_{n,2})
- \dt_n \nu_{1} \langle \scrF_{n,1},v_{n,1} \rangle
- \dt_n \nu_{2} \langle \scrF_{n,2} , v_{n,2} \rangle = \frac{1}{2}|u_{n}|^2
\end{equation}
where $\nu_{1} = \nu_{11} + \nu_{12}$, $\nu_{2} = \nu_{22}$, and $\mathcal{Q} :
\mathbb{H}^3 \to \mathbb{R}$ is a quadratic form given by
\begin{equation}\label{QformTwoStage}
  \mathcal{Q}(u_n,v_{n,1},v_{n,2}) = \delta_{1} |v_{n,1} - u_n|^2 + \delta_{2}
|v_{n,2} - v_{n,1}|^2
  - \frac{\nu_{12}}{a_{11}} (v_{n,1} - u_n, v_{n,2} - v_{n,1}) .
\end{equation}
\end{corollary}
\begin{proof} Exploiting the bilinearity of the duality pairing we have that
\[
  \nu_{12} \langle \scrF_{n,1}, v_{n,2} \rangle = \nu_{12} \langle \scrF_{n,1},
v_{n,1} \rangle
  + \nu_{12} \langle \scrF_{n,1}, v_{n,2} - v_{n,1} \rangle .
\]
We use this identity to replace the last term on the right hand side of
\eqref{eq:protoenergy}. After reorganizing the terms we get
\begin{multline}
\label{AddSubtractTrick}
  \frac12 |u_{n+1}|^2 + \delta_1 |v_{n,1} - u_n|^2 + \delta_2 |v_{n,2} -
v_{n,1}|^2 - \dt_n \left[ (\nu_{11} + \nu_{12}) \langle \scrF_{n,1}, v_{n,1}
\rangle + \nu_{22} \langle \scrF_{n,2}, v_{n,2} \rangle \right] = \\
  \frac12 |u_n|^2 + \dt_n \nu_{12} \langle \scrF_{n,1}, v_{n,2} - v_{n,1}
\rangle.
\end{multline}

Taking the inner product of the first stage of \eqref{eq:DIRKappliedtoPDE} with
$\nu_{12}(v_{n,2}- v_{n,1})$ we get that
\[
  \dt_n \nu_{12} \langle \scrF_{n,1}, v_{n,2} - v_{n,1} \rangle =
\frac{\nu_{12}}{a_{11}} (v_{n,1} - u_n, v_{n,2} - v_{n,1}) .
\]
Inserting this identity into the right hand side of \eqref{AddSubtractTrick}
and reorganizing the terms yields the desired result.
\end{proof}

We note that the discrete energy-balance laws \eqref{eq:protoenergy} and
\eqref{GeneralEnergyBal} do not carry much practical value unless the sign of
the coefficients is correct, and the quadratic form $\mathcal Q$ is
nonnegative. However, if this is the case, our schemes will have all requisite
stability properties. We encode this in the following definition.

\begin{definition}[remarkable stability I]\label{RemarkableDefTwoStage}
We say that the DIRK scheme \eqref{eq:TheDIRK} with $s=2$, $\bv{A}$ lower
triangular and with positive diagonal entries is
remarkably stable if the following conditions hold
\begin{align*}
  \delta_1 \geq 0, \quad
  \delta_2 \geq 0 , \quad
  \nu_1 = \nu_{11} + \nu_{12} > 0 , \quad
  \nu_2 = \nu_{22} > 0, \quad
\end{align*}
with $\delta_1$, $\delta_2$, $\nu_{11}$, $\nu_{22}$, and $\nu_{12}$ defined in
\eqref{DisipCoeffDirk22}, and the quadratic form $\mathcal{Q}$, introduced in
\eqref{QformTwoStage}, is nonnegative definite.
\end{definition}

Remarkable stability defines an exceptional class of schemes for which the off-diagonal term $\langle \scrF_{n,1}, v_{n,2} \rangle$ on the right hand side
of \eqref{eq:protoenergy} can always be absorbed into artificial damping terms
regardless of the nature of $\mathcal{F}$ (coercive, linear, nonlinear,
degenerate, skew symmetric, etc.).

The following result provides sufficient, easy to check, conditions for
$\mathcal Q$ to be nonnegative.

\begin{proposition}[nonnegativity] \label{PropPositiv22}
Assume that, in the setting of Corollary~\ref{PropRewritten22}, we have
\[
  \delta_1 \geq 0 , \qquad
  \delta_2 \geq 0, \qquad
  \left|\frac{\nu_{12}}{a_{11}} \right| \leq 2 \sqrt{\delta_1} \sqrt{\delta_2} .
\]
Then, the quadratic from $\mathcal{Q}$, introduced in \eqref{QformTwoStage} in
nonnegative definite.
\end{proposition}
\begin{proof}
If $\delta_1 \geq 0$ and $\delta_2 \geq 0$ we have that
\begin{align*}
  \delta_{1}|v_{n,1} - u_n|^2 \pm 2 \sqrt{\delta_{1}} \sqrt{\delta_{2}} (v_{n,1}
- u_n, v_{n,2} - v_{n,1}) + \delta_{2} |v_{n,2} - v_{n,1}|^2 \geq 0 .
\end{align*}
Therefore,
\begin{align*}
\pm 2 \sqrt{\delta_{1}} \sqrt{\delta_{2}} (v_{n,1} - u_n, v_{n,2} - v_{n,1})
\leq \delta_{1}|v_{n,1} - u_n|^2 + \delta_{2} |v_{n,2} - v_{n,1}|^2.
\end{align*}
This, in particular, implies that
\begin{align*}
\beta (v_{n,1} - u_n, v_{n,2} - v_{n,1}) \leq \delta_{1}|v_{n,1} - u_n|^2 +
\delta_{2} |v_{n,2} - v_{n,1}|^2,
\end{align*}
for all $\beta \in \mathbb{R}$ satisfying $|\beta| \leq 2 \sqrt{\delta_{1}}
\sqrt{\delta_{2}}$.
\end{proof}

\begin{remark}[nonnegativity]
We comment that the condition of Proposition~\ref{PropPositiv22} is only
sufficient. Necessary and sufficient conditions are obtained by looking at the
spectrum of the coefficient matrix of the quadratic form $\mathcal{Q}$. In this
case we have
\[
  \bv{Q} = \begin{bmatrix}
           \delta_1 & -\delta_1 - \frac{\nu_{12}}{a_{11}} &
\frac{\nu_{12}}{a_{11}}\\
           -\delta_1 - \frac{\nu_{12}}{a_{11}} & \delta_1 + \delta_2 +
\frac{2\nu_{12}}{a_{11}}& -\delta_2 - \frac{\nu_{12}}{a_{11}}\\
           \frac{\nu_{12}}{a_{11}} & -\delta_2 - \frac{\nu_{12}}{a_{11}}   &
\delta_2
         \end{bmatrix} .
\]
Lengthy and painful, but trivial, computations reveal that
\[
  \sigma(\bv{Q})= \left\{ 0, \delta_1 + \delta_2 + \frac{\nu_{12}}{2a_{11}} \pm
  \sqrt{
    \delta_1^2 - \delta_1 \delta_2 + \delta_2^2 + (\delta_1 + \delta_2)
\frac{\nu_{12}}{a_{11}} + \left(\frac{\nu_{12}}{a_{11}} \right)^2
  } \right\}.
\]
\end{remark}

\begin{remark}[computational aspects] Two important aspects of computational
practice are time-adaptivity and nonlinear solver tolerances. Assume that we
are using a remarkably stable scheme and that we are able to solve for the
stages $\{v_{n,1}, v_{n,2}\}$ exactly (i.e., to machine accuracy). As a
consequence, we obtain
\begin{align*}
  \frac{1}{2}\left( |u_{n+1}|^2 - |u_{n}|^2 \right) - \dt_n \nu_{1} \langle
\scrF_{n,1},v_{n,1} \rangle - \dt_n \nu_{2} \langle \scrF_{n,2},v_{n,2} \rangle
= - \mathcal{Q}(u_n,v_{n,1},v_{n,2}) \leq 0.
\end{align*}
We note that the functional $\mathcal{Q}(u_n,v_{n,1},v_{n,2})$ gives us
exactly how much numerical dissipation occurred from time instance $t_n$ to
$t_{n+1}$.  In this context, the value of the quadratic form
$\mathcal{Q}(u_n,v_{n,1},v_{n,2})$ may be used as the foundation
for the development of an heuristic error indicator in order to drive a
time-adaptive process. We note that using numerical dissipation
as an a posteriori estimator in order to select the timestep size is not a new
idea, see for instance \cite[Remark 3.4]{Makri2006} and references therein.

However, in the previous paragraph, we made a very strong assumption: we can
solve the nonlinear problems at each stage ``exactly'', which is rarely ever
true. In that such context, we may not be able to use $\mathcal{Q}$ in order to
quantify numerical dissipation. Let $\{\tilde v_{n,1}, \tilde v_{n,2},\tilde
u_{n+1}\}$ represent our ``inexact approximations'' of the first stage
$v_{n,1}$, second stage $v_{n,2}$ and final solution $u_{n+1}$ respectively,
then we can always define the functional
\begin{align*}
  \eta( \tilde v_{n,1}, \tilde v_{n,2},\tilde u_{n+1} ) = \frac{1}{2}\left(
|\tilde u_{n+1}|^2 - |u_{n}|^2 \right) - \dt_n \nu_{1} \langle \tilde
\scrF_{n,1}, \tilde v_{n,1} \rangle - \dt_n \nu_{2} \langle \tilde
\scrF_{n,2},
\tilde v_{n,2} \rangle,
\end{align*}
where, for $i=1,2$, $\tilde\scrF_{n,i}$ has the expected meaning. Indeed, if
$\eta$ is sufficiently negative, we may argue that the scheme implementation
exhibits numerically dissipative behavior. Otherwise, the numerical tolerances
and/or the timestep need to be reduced, and the whole time-step should be solved
again.
\end{remark}

\subsection{Examples of two-stage remarkably stable
schemes}\label{sub:ExampleTwoStageRemarkable}
Let us now explore whether the schemes of Section~\ref{subsub:TwoStageList} are
remarkably stable.

\subsubsection{The Butcher-Burrage DIRK22 scheme}\label{RemarkBB}
We consider the two-stage scheme with tableau given in
\eqref{ButcherBurrageTableau}.
Using formulas \eqref{Lambda12general} and
\eqref{DisipCoeffDirk22}, for the case of $\gamma_1$ we get
\begin{align*}
\lambda_1 = -\tfrac{\sqrt{2}}{2}\, , \
\lambda_2 = 1+\tfrac{\sqrt{2}}{2} \, , \
\delta_1 = \tfrac{1}{2} \, , \
\delta_2 = \tfrac{1}{4} \, , \
\nu_{11} = 1 - \tfrac{\sqrt{2}}{2} \, , \
\nu_{22} = \tfrac{1}{2} \, , \
\nu_{12} = -\tfrac{1}{2} + \tfrac{\sqrt{2}}{2} \, , \
\end{align*}
which leads to the following properties
\begin{align*}
\nu_{1} = \tfrac{1}{2} \ , \ \
\nu_{2} = \tfrac{1}{2} \ , \ \
\tfrac{\nu_{12}}{a_{11}} = \tfrac{\sqrt{2}}{2} \ , \ \
2 \sqrt{\delta_1} \sqrt{\delta_2} = \tfrac{\sqrt{2}}{2} \, .
\end{align*}
We conclude that Butcher-Burrage scheme \eqref{ButcherBurrageTableau} is
remarkably-stable for the case of $\gamma_1$. For the case of $\gamma_2$ we
obtain:
%
%
%
\begin{align*}
\lambda_1 = \tfrac{\sqrt{2}}{2}\, , \
\lambda_2 = 1-\tfrac{\sqrt{2}}{2} \, , \
\delta_1 = \tfrac{1}{2} \, , \
\delta_2 = \tfrac{1}{4} \, , \
\nu_{11} = \tfrac{\sqrt{2}}{2} + 1 \, , \
\nu_{22} = \tfrac{1}{2} \, , \
\nu_{12} = -\tfrac{\sqrt{2}}{2} - \tfrac{1}{2} \, , \
\end{align*}
leading to:
\begin{align*}
\nu_{1} = \tfrac{1}{2} \ , \ \
\nu_{2} = \tfrac{1}{2} \ , \ \
\tfrac{\nu_{12}}{a_{11}} = - \tfrac{\sqrt{2}}{2} \ , \ \
2 \sqrt{\delta_1} \sqrt{\delta_2} = \tfrac{\sqrt{2}}{2} \, .
\end{align*}
We conclude that Butcher-Burrage scheme is remarkably stable for the case of
$\gamma_2$ as well.

\subsubsection{The Crouzeix DIRK23 scheme}\label{RemarkCro}
We consider the two-stage scheme with tableau described in
\eqref{CrouzeixTableau}. Using formulas \eqref{Lambda12general} and
\eqref{DisipCoeffDirk22} we get
\begin{align*}
\lambda_1 = \tfrac{3 \sqrt{3}}{2} - \tfrac{3}{2} \, , \
\lambda_2 = \tfrac{3}{2} - \tfrac{\sqrt{3}}{2}\, , \
\delta_1 = \sqrt{3} - \tfrac{3}{2} \, , \
\delta_2 = \tfrac{\sqrt{3}}{4} \, , \
\nu_{11} = 1 \, , \
\nu_{22} = \tfrac{1}{2} \, , \
\nu_{12} = -\tfrac{1}{2} \, , \
\end{align*}
which leads to the following values
\begin{align*}
\nu_1 = \tfrac{1}{2} \, , \ \
\nu_2 = \tfrac{1}{2} \, , \ \
\tfrac{\nu_{12}}{a_{11}} = \tfrac{\sqrt{3}}{2} - \tfrac{3}{2} \ , \ \
2 \sqrt{\delta_1} \sqrt{\delta_2} = \tfrac{3}{2} - \tfrac{\sqrt{3}}{2} \, ,
\end{align*}
which allows us to conclude that the Crouzeix DIRK23 scheme
\eqref{CrouzeixTableau} is remarkably stable.

\begin{remark}[quadrature]
We note that the pair of collocation points $\{c_1,c_2\}$ and weights
$\{\nu_{1},\nu_{2}\}$ associated to the Crouzeix DIRK23 scheme define a
quadrature rule on the interval $[0,1]$ that is exact for polynomials of degree
at most three. This might facilitate the derivation of the ``equation satisfied
by the error'' and the development of an a priori error analysis without the
need of defining a quadratic in time piecewise polynomial reconstruction.
\end{remark}

\subsubsection{Alexander's DIRK22 scheme}\label{RemarkAlex22}
We consider the two-stage scheme with tableau described in
\eqref{AlexDIRK22tableu}. Note that in this case $b_1 = a_{21}$ and $b_2 =
a_{22}$. Using formulas \eqref{Lambda12general} and \eqref{DisipCoeffDirk22} we
get
\begin{align*}
\lambda_1 = 0 \, , \
\lambda_2 = 1 \, , \
\delta_1 = \tfrac{1}{2} \, , \
\delta_2 = \tfrac{1}{2} \, , \
\nu_{11} = 1 - \tfrac{\sqrt{2}}{2} \, , \
\nu_{22} = 1 - \tfrac{\sqrt{2}}{2} \, , \
\nu_{12} = \sqrt{2} -1\, , \
\end{align*}
which leads to the following properties
\begin{align*}
\nu_1 = \tfrac{\sqrt{2}}{2} \, , \
\nu_2 = 1 - \tfrac{\sqrt{2}}{2} \, , \ \
\sigma(\bv{Q}) &= \left\{0, \frac{1- \sqrt2}2, \frac{3(1+\sqrt2)}2 \right\}
\approx \{0.,-0.2017107,3.62132\}.
\end{align*}
This allows us to conclude that Alexander's DIRK22 scheme
\eqref{AlexDIRK22tableu} is \textbf{not} remarkably stable. \\

We comment that, in principle, the fact that a DIRK scheme is not remarkably-stable does not
mean that it should not be used. As detailed in Appendix \ref{AlexanderAppendix},
if a scheme is not remarkably stable it is, in principle, not possible to guarantee
energy-stability if the operator $\mathcal{A}$ is skew-symmetric (i.e.
$\langle\mathcal{A}(u),v\rangle = - \langle\mathcal{A}(v),u\rangle$).
For this reason, if a scheme is not remarkably stable, its utility may be
limited to linear, coercive, self-adjoint problems; see, again, Appendix~\ref{AlexanderAppendix}.

\subsubsection{Kraaijevanger-Spijker's DIRK22 scheme}\label{RemarkKra}
We consider the two-stage scheme with tableau described in
\eqref{KraaijDIRK22}.
Using formulas \eqref{Lambda12general} and \eqref{DisipCoeffDirk22} we get
\begin{align*}
\lambda_1 = - \tfrac{1}{4} \, , \
\lambda_2 = \tfrac{3}{4} \, , \
\delta_1 = \tfrac{3}{8} \, , \
\delta_2 = \tfrac{15}{32} \, , \
\nu_{11} = \tfrac{7}{16} \, , \
\nu_{22} = \tfrac{3}{2} \, , \
\nu_{12} = \
- \tfrac{15}{16} \, , \
\end{align*}
which leads to the following properties
\begin{align*}
\nu_1 = - \tfrac{1}{2} \, , \
\nu_2 = \tfrac{3}{2} \, .
\end{align*}
This allows us to conclude DIRK22 scheme \eqref{KraaijDIRK22} is \textbf{not}
remarkably stable. Regardless of the sign of $\mathcal{Q}$, the coefficient
$\nu_1$ is negative. For these reasons, the applicability of such scheme to
evolutionary PDEs, even in the case of linear, positive, and symmetric operators, is
rather limited.


\section{Three-stage schemes}\label{sec:ThreeStage}

The goal of this section is similar to that of the previous one, that is, we
will introduce a class of \emph{remarkably stable} DIRK schemes, see
Definition~\ref{def:RemarkableThreeStage}, that satisfy a dissipative discrete
energy-balance. Conceptually, this section is no different than the previous
one. The logical steps to achieve our goal are exactly the same. The algebraic
manipulations, however, are more involved and tedious. We urge the reader to
take advantage of a computer algebra system.

\subsection{General discrete energy-balance
laws}\label{sub:EnergyBalanceThreeStage}

We, once again, specialize Proposition~\ref{prop:extrapolateDIRK} to the case of
a three-stage DIRK.

\begin{corollary}[three-stage extrapolation]
\label{cor:ThreeStageExtrapolate}
Let the RK scheme \eqref{eq:TheDIRK} be such that $s=3$, $\bv{A}$ is lower
triangular, and with positive diagonal entries. Then, for all $n\geq 0$,
\begin{equation}
\label{FinalSolLinearDirk33}
  u_{n+1} = \left( 1- \lambda_1 - \lambda_2 - \lambda_3 \right) u_n + \lambda_1
v_{n,1} + \lambda_2 v_{n,2} + \lambda_3 v_{n,3},
\end{equation}
with
\begin{equation}
\label{LambdaDIRK33}
  \lambda_1 =  \frac{b_1}{a_{11}} - \frac{b_2 a_{21}}{a_{11} a_{22}} - \frac{b_3
a_{31}}{a_{11}a_{33}} + \frac{b_3 a_{32} a_{21}}{a_{11} a_{22} a_{33}},  \quad
  \lambda_2 = \frac{b_2}{a_{22}} - \frac{b_3 a_{32}}{a_{22}a_{33}}, \quad
  \lambda_3 = \frac{b_3}{a_{33}}.
\end{equation}
\end{corollary}
\begin{proof}
It follows from a direct computation. The fact that $\bv{A}$ is lower triangular
 simplifies these.
\end{proof}

Next we obtain some more useful identities.

\begin{lemma}[more useful identities]
Let $s=3$. For any $\frakN \in \polN$, any partition $\calP_\frakN$, and all
$n \in \{ 0, \ldots, \frakN -1 \}$ we have that $\{u_n,
v_{n,1},v_{n,2},v_{n,3},u_{n+1} \}$, coming from \eqref{eq:DIRKappliedtoPDE}
with $a_{ii}>0$, satisfy \eqref{eq:id1}---\eqref{eq:id5}, and, additionally,
\begin{multline}
\label{eq:id6}
  \frac12 \left( | v_{n,3} |^2 + | v_{n,3} - v_{n,2} |^2 - | v_{n,2} |^2\right)
= \\ \dt_n \left[ (a_{31}-a_{21})\langle \scrF_{n,1},v_{n,3} \rangle +
(a_{32}-a_{22}) \langle \scrF_{n,2}, v_{n,3} \rangle + a_{33} \langle
\scrF_{n,3}, v_{n,3} \rangle \right],
\end{multline}
\begin{equation}
\label{eq:id7}
  | v_{n,3} |^2 -(u_n,v_{n,3}) = \tau_n \left[ a_{31}\langle \scrF_{n,1},
v_{n,3} \rangle + a_{32}\langle \scrF_{n,2}, v_{n,3} \rangle + a_{33} \langle
\scrF_{n,3}, v_{n,3} \rangle \right],
\end{equation}
\begin{equation}
\label{eq:id8}
  | v_{n,3} |^2 -(v_{n,2},v_{n,3}) = \tau_n \left[ (a_{31}-a_{21}) \langle
\scrF_{n,1}, v_{n,3} \rangle + (a_{32}- a_{22})\langle \scrF_{n,2}, v_{n,3}
\rangle + a_{33} \langle \scrF_{n,3}, v_{n,3} \rangle \right],
\end{equation}
\begin{equation}
\label{eq:id9}
  | v_{n,3} |^2 -(v_{n,1},v_{n,3}) = \tau_n \left[ (a_{31}-a_{11}) \langle
\scrF_{n,1}, v_{n,3} \rangle + a_{32}\langle \scrF_{n,2}, v_{n,3} \rangle +
a_{33} \langle \scrF_{n,3}, v_{n,3} \rangle \right],
\end{equation}
\end{lemma}
\begin{proof}
We, first of all, note that owing to the fact that we are dealing with a DIRK
scheme we can compute sequentially the states. Therefore, identities
\eqref{eq:id1}---\eqref{eq:id5} remain to hold. The new identities can be
obtained as before. Identity \eqref{eq:id6} comes from testing the third stage
by $v_{n,3}$ and applying the polarization identity
\eqref{FundIdentity}. Identity \eqref{eq:id7} comes from testing the third
stage by $v_{n,3}$. Identity \eqref{eq:id8} comes from testing the third stage
by $v_{n,2}$. Combining the first and third stage, and testing the result with
$v_{n,3}$ gives \eqref{eq:id9}.
\end{proof}

With these identities at hand we can prove an analogue of \eqref{eq:protoenergy}
for the case $s=3$.

\begin{theorem}[discrete energy identity I]\label{thm:energyidentity3stage}
Let $s=3$. For any $\frakN \in \polN$, any partition $\calP_\frakN$, and any
$n \in \{0, \ldots, \frakN-1\}$ we have that the solution of
\eqref{eq:DIRKappliedtoPDE}, with $a_{ii}>0$, satisfies
\begin{multline}
\label{Dirk33GeneralEnergyID}
  \frac12 | u_{n+1} |^2 + \sum_{i=1}^3 \delta_i | v_{n,i} - v_{n,i-1} |^2
  - \tau_n \sum_{i=1}^3 \nu_{ii} \langle \scrF_{n,i} ,v_{n,i} \rangle = \\
  \frac12 | u_n |^2 + \tau_n \left[ \nu_{12} \langle \scrF_{n,1}, v_{n,2}
\rangle + \nu_{13} \langle \scrF_{n,1}, v_{n,3}\rangle + \nu_{23} \langle
\scrF_{n,2}, v_{n,3} \rangle \right],
\end{multline}
where $v_{n,0} = u_n$,
\begin{equation}
\label{EnergyCoeff33}
  \begin{aligned}
    \delta_1 &= \frac12(\lambda_1 + \lambda_2 + \lambda_3)(2-\lambda_1 -
\lambda_2 - \lambda_3) , \\
    \delta_2 &= \frac12(\lambda_2 + \lambda_3)(2-\lambda_2 - \lambda_3), \\
    \delta_3 &= \frac12\lambda_3 (2-\lambda_3), \\
    \nu_{11} &=  a_{11} (\lambda_1 (1 - \lambda_2 - \lambda_3) + (2 - \lambda_2
- \lambda_3) (\lambda_2 + \lambda_3)), \\
    \nu_{22} &=  a_{22} (\lambda_2 (1 - \lambda_3) + (2 - \lambda_3) \lambda_3),
\\
    \nu_{33} &= a_{33}\lambda_3, \\
    \nu_{12} &=  [a_{21} (\lambda_2 (1- \lambda_3) + (2 - \lambda_3) \lambda_3)
+ a_{11} (\lambda_2 (-2 + \lambda_1 + \lambda_2) + 2 (-1 + \lambda_2) \lambda_3
+ \lambda_3^2)], \\
    \nu_{13} &=  \lambda_3 (a_{31} + a_{11} \lambda_1 + a_{21} (-2 + \lambda_2 +
\lambda_3)), \\
    \nu_{23} &=  \lambda_3 (a_{32} + a_{22} (-2 + \lambda_2 + \lambda_3)),
  \end{aligned}
\end{equation}
and $\lambda_1, \lambda_2, \lambda_3$ were defined in
Corollary~\ref{cor:ThreeStageExtrapolate}.
\end{theorem}
\begin{proof}
We once again invite the reader to launch their favorite computer algebra
software, as the proof of this result merely entails lengthy and tortuous, but
trivial, computations. One merely has to take the inner product of the
extrapolation identity of Corollary~\ref{cor:ThreeStageExtrapolate} with itself,
and use in the result identities \eqref{eq:id1}---\eqref{eq:id5} and
\eqref{eq:id6}---\eqref{eq:id9}.
\end{proof}

\begin{remark}[consistency check]
Direct computation shows that $\nu_{11} +\nu_{22} + \nu_{33} + \nu_{12} +
\nu_{13} + \nu_{23} = b_1 + b_2 + b_3$. Assuming that the scheme satisfies
second-order conditions, see Appendix~\ref{AppOrderCond}, we have
that $\nu_{11} +\nu_{22} + \nu_{33} + \nu_{12} + \nu_{13} + \nu_{23} = 1$.
\end{remark}

As we did in the previous section for the case of two-stage schemes, we
rewrite \eqref{Dirk33GeneralEnergyID} in a form that will resemble a dissipative
discrete energy law.

\begin{corollary}[discrete energy identity II]
The discrete energy identity \eqref{Dirk33GeneralEnergyID} can be rewritten as
\begin{equation}
\label{EnergyID33poly}
  \frac12 |u_{n+1}|^2 + \mathcal{Q}(u_n, v_{n,1}, v_{n,2}, v_{n,3}) - \dt_n
\sum_{i=1}^3 \nu_i \langle \scrF_{n,i},v_{n,i} \rangle = \frac 12 |u_n|^2,
\end{equation}
where, for $i = 1, \ldots, 3$,
\begin{equation}
\label{eq:NuSingleIndexThreeStage}
  \nu_i = \nu_{ii} + \sum_{j=i+1}^3 \nu_{ij},
\end{equation}
the quadratic form $\mathcal Q$ is given by
\begin{equation}
\label{Qform33}
  \begin{aligned}
    \mathcal{Q}(u_n,v_{n,1},v_{n,2},v_{n,3}) &= \delta_{1} |v_{n,1} - u_n|^2 +
\delta_{2} |v_{n,2} - v_{n,1}|^2
    + \delta_{3} |v_{n,3} - v_{n,2}|^2 \\&- \frac{\nu_{12}}{a_{11}} (v_{n,1} -
u_n, v_{n,2} - v_{n,1}) \\
    & - \frac{\nu_{13}}{a_{11}} (v_{n,1} - u_n, v_{n,3} - v_{n,1})
    - \frac{\nu_{23}}{a_{22}}(v_{n,2} - u_n, v_{n,3} - v_{n,2}) \\
    &+ \frac{\nu_{23} a_{21}}{a_{11} a_{22}} (v_{n,1} - u_n, v_{n,3} - v_{n,2})
,
  \end{aligned}
\end{equation}
and $\delta_1$, $\delta_2$, $\delta_3$, $\nu_{11}$, $\nu_{22}$, $\nu_{33}$,
$\nu_{12}$, $\nu_{13}$ and $\nu_{23}$ are defined in \eqref{EnergyCoeff33}.
\end{corollary}
\begin{proof}
We follow the ideas that led to \eqref{GeneralEnergyBal} in the case of two
stage schemes. We exploit the fact that the matrix $\bf{A}$ is lower triangular,
together with the bilinearity of the duality pairing to get, using the equations
of the stages,
\begin{align*}
  \tau_n\nu_{12} \langle \scrF_{n,1}, v_{n,2} \rangle &=
    \tau_n\nu_{12} \langle \scrF_{n,1}, v_{n,1} \rangle + \tau_n\nu_{12} \langle
\scrF_{n,1}, v_{n,2} - v_{n,1} \rangle \\
    &= \tau_n\nu_{12} \langle \scrF_{n,1}, v_{n,1} \rangle +
\frac{\nu_{12}}{a_{11}}(v_{n,1}-u_n,v_{n,2}-v_{n,1}), \\
  \tau_n\nu_{13} \langle \scrF_{n,1}, v_{n,3} \rangle &=
    \tau_n\nu_{13} \langle \scrF_{n,1}, v_{n,1} \rangle + \tau_n\nu_{13} \langle
\scrF_{n,1}, v_{n,3} - v_{n,1} \rangle \\
    &= \tau_n\nu_{13} \langle \scrF_{n,1}, v_{n,1} \rangle +
\frac{\nu_{13}}{a_{11}}(v_{n,1}-u_n,v_{n,3}-v_{n,1}), \\
  \tau_n \nu_{23} \langle\scrF_{n,2}, v_{n,3} \rangle &=
    \tau_n \nu_{23} \langle\scrF_{n,2}, v_{n,2} \rangle + \tau_n
\nu_{23}\langle\scrF_{n,2}, v_{n,3} - v_{n,2}\rangle \\
    &= \tau_n \nu_{23} \langle\scrF_{n,2}, v_{n,2} \rangle +
\frac{\nu_{23}}{a_{22}} (v_{n,2} - u_n, v_{n,3} - v_{n,2}) -
\frac{a_{21}\nu_{23}}{a_{11}a_{22}} (v_{n,1} - u_n, v_{n,3}-v_{n,2}).
\end{align*}

We use the previous identities in order to rewrite the energy identity
\eqref{Dirk33GeneralEnergyID} as follows
\begin{align}
\begin{split}\label{EnergyID33polyAlmostPoly}
 \frac12 | u_{n+1} |^2  &+ \sum_{i=1}^3 \left( \delta_i | v_{n,i} - v_{n,i-1}
|^2 - \tau_n \left(\nu_{ii}+\sum_{j=i+1}^3 \nu_{ij} \right) \langle \scrF_{n,i},
v_{n,i} \rangle \right) = \frac12 |u_n |^2\\
& + \frac{\nu_{12}}{a_{11}} (v_{n,1}-u_n,v_{n,2}-v_{n,1}) +
\frac{\nu_{13}}{a_{11}} (v_{n,1}-u_n,v_{n,3}-v_{n,1}) \\
& + \frac{\nu_{23}}{a_{22}} (v_{n,2} - u_n, v_{n,3} - v_{n,2} ) -
\frac{a_{21}\nu_{23}}{a_{11}a_{22}} (v_{n,1} - u_n,v_{n,3}-v_{n,2}) .
\end{split}
\end{align}
Finally, identity \eqref{EnergyID33poly} follows by reorganizing the terms
in \eqref{EnergyID33polyAlmostPoly}.
\end{proof}

Once again, the practical value of identity \eqref{EnergyID33poly} rests on
the quadratic form $\mathcal{Q}$, and whether or not it is nonnegative definite. As in the
case of two-stage schemes, we introduce the notion of remarkably stable
three-stage schemes.

\begin{definition}[remarkable stability]\label{def:RemarkableThreeStage}
We will say that the DIRK scheme \eqref{eq:TheDIRK} with $s=3$, $\bv{A}$ lower
triangular, and with positive diagonal entries is
remarkably stable if
\[
\delta_1 \geq 0 , \quad
\delta_2\geq 0 , \quad
\delta_3 \geq 0, \quad
\nu_1 > 0 , \quad
\nu_2 > 0 , \quad
\nu_3 > 0 ,
\]
where these coefficients were defined in \eqref{EnergyCoeff33} and
\eqref{eq:NuSingleIndexThreeStage}, and, in addition, the quadratic form
$\mathcal{Q}$, defined in \eqref{Qform33}, is nonnegative definite.
\end{definition}

\begin{remark}[nonnegativity] Since $\mathcal Q$, as defined in
\eqref{Qform33}, can always be expanded in terms of its monomial coefficients,
the nonnegativity of the quadratic form $\mathcal Q$ can be verified by
examining the eigenvalues of the corresponding coefficient matrix
$\bv{Q}=[q_{ij}]_{i,j=1}^4$, which read
\begin{align*}
  q_{11} &= \delta_1 , &
  q_{12} = q_{21} &= -\delta_1 - \frac{\nu_{12}}{a_{11}} -
\frac{\nu_{13}}{a_{11}}, \\
  q_{13} = q_{31} &= \frac{\nu_{12}}{a_{11}} - \frac{\nu_{23}}{a_{22}} +
\frac{a_{21}\nu_{23}}{a_{11} a_{22}}, &
  q_{14} = q_{41} &= \frac{\nu_{13}}{a_{11}} + \frac{\nu_{23}}{a_{22}} -
\frac{a_{21}\nu_{23}}{a_{11} a_{22}}, \\
  q_{22} &= \delta_1 + \delta_2 + \frac{2\nu_{12}}{a_{11}} +
\frac{2\nu_{13}}{a_{11}}, &
  q_{23} = q_{32} &= - \delta_2 - \frac{\nu_{12}}{a_{11}} - \frac{a_{21}
\nu_{23}}{a_{11} a_{22}}, \\
  q_{24} = q_{42} &= -\frac{\nu_{13}}{a_{11}} + \frac{a_{21} \nu_{23}}{a_{11}
a_{22}}, &
  q_{33} &= \delta_2 + \delta_3 + \frac{2\nu_{23}}{a_{22}}, \\
  q_{34} = q_{43} &= - \delta_3 - \frac{\nu_{23}}{a_{22}}, &
  q_{44} &= \delta_2.
\end{align*}
\end{remark}

\subsection{Examples of three-stage remarkably stable
schemes}\label{sub:ExampleThreeStageRemarkable}
Let us now investigate the schemes of Section~\ref{subsub:ThreeStageList}
for remarkable stability.

\subsubsection{Crouzeix-Raviart DIRK34 scheme}

We consider the scheme described by tableau \eqref{CrouRavTableau} with
$\gamma_1$ as defined in \eqref{GammasCroRav}, also known as the
Crouzeix-Raviart scheme. In this case we have that (with 20 digits of accuracy
for the sake of computational utility):
\begin{align*}
[\lambda_1, \lambda_2, \lambda_3] &\approx
[0.44562240728771388189, 1.0641777724759121408, 0.12061475842818323189] \\
[\delta_1, \delta_2, \delta_3] &\approx
[0.30128850285230865863, 0.48292586026102947830, 0.11334079845283873290] \\
[\nu_{11}, \nu_{22}, \nu_{33}] &\approx
[0.94409386961162504966, 1.2422271989685591552, 0.12888640051572042236] \\
[\nu_{12}, \nu_{13}, \nu_{23}] &\approx
[- 1.1863210685801842049, 0.37111359948427957763, -0.5]
\end{align*}
which leads to:
\begin{align*}
\nu_1 \approx 0.12888, \quad
\nu_2 \approx 0.74222, \quad
\nu_3 \approx 0.12888 .
\end{align*}

Up to 20 digits of accuracy we find that the eigenvalues of the matrix
$\bv{Q}$ are approximately
\[
  \{0.564309,0, 0, 0\},
\]
hinting at the fact that $\mathcal Q$ is nonnegative definite. A reduction to
row echelon form of this matrix gives us that it has exactly three rows that
consist only of zeros. This means this matrix has three zero eigenvalues, and
the scheme is remarkably stable.

To conclude with this example, we mention that neither the case of $\gamma_2$
nor $\gamma_3$, defined in \eqref{GammasCroRav}, lead to remarkably stable
schemes. In particular, they lead to $\delta_1, \delta_2, \delta_3 < 0$.

\subsubsection{Alexander's DIRK33 scheme}\label{SubAlex33}

We consider the scheme described by tableau \eqref{DIRK33Lstable}. In this
context we have:
\begin{align*}
[\lambda_1, \lambda_2, \lambda_3] &\approx
[0, 0, 1] \\
[\delta_1, \delta_2, \delta_3] &\approx
[0.5, 0.5, 0.5] \\
[\nu_{11}, \nu_{22}, \nu_{33}] &\approx
[0.435866, 0.435866, 0.435866] \\
[\nu_{12}, \nu_{13}, \nu_{23}] &\approx
[- 0.153799, 0.926429, -1.080229]
\end{align*}
leading to
\begin{align*}
\nu_1 \approx 1.2084966 \ , \
\nu_2 \approx -0.644363 \ \text{and} \
\nu_3 \approx 0.43586652 \, .
\end{align*}
We conclude that, regardless of the spectrum of $\bv{Q}$, Alexander's
DIRK33 scheme is \textbf{not} remarkably stable. Just like Alexander's
DIRK22 scheme, defined in tableau \eqref{AlexDIRK22tableu} and considered in
section~\ref{RemarkAlex22}, Alexander's DIRK33 scheme is not unconditionally stable for skew-symmetric problems. We delve into these details in Appendix \ref{AlexanderAppendix}.

\section{Bochner-type norm estimates}
\label{sec:AisCoercive}

In this section we introduce a more stringent set of assumptions on the
mapping $\calA$. These will be invoked if we wish to prove a priori bounds in
the Bochner-type norm $L^p(0,\tf;\polV)$, for some $p>1$, on the solution of
\eqref{eq:TheODE}. These assumptions, in addition, will allow us to obtain a
priori bounds for the time derivative of the solution in the Bochner-type norm
$L^{q'}(0,\tf;\mathbb{V}^*)$ for some $q>1$. These two estimates are enough to
establish compactness of the family of approximate solutions via the
well-celebrated Aubin-Lions compactness lemma.

\begin{lemma}[Aubin-Lions]\label{AubinLionsLemma}
Let $\mathbb{X}$, $\mathbb{Y}$ and $\mathbb{Z}$ be three Banach spaces such
that $\mathbb{X} \subseteq \mathbb{Y} \subseteq \mathbb{Z}$. Assume that
$\mathbb{X}$ is compactly embedded in $\mathbb{Y}$, and $\mathbb{Y}$ is
continuously embedded in
$\mathbb{Z}$. Then, for $1 \leq p,q \leq +\infty$ we define
\begin{align*}
  \mathbb{U} = \left\{ u \in L^p(0,\tf;\mathbb{X}) \ \middle| \ \frac{\diff
u}{\diff t} \in L^{q'}(0,\tf;\mathbb{Z}) \right\}.
\end{align*}
Then
\begin{enumerate}[$\bullet$]
  \item If $p < +\infty$ the embedding of
$\mathbb{U}$ into $L^p(0,\tf;\mathbb{Y})$ is compact.

  \item If $p = +\infty$ and $q <\infty$ the embedding of $\mathbb{U}$ into
$\mathcal{C}([0,\tf];\mathbb{Y})$ is compact.
\end{enumerate}
\end{lemma}

\begin{remark}[references]
The origins of Lemma~\ref{AubinLionsLemma} go back to
\cite{Aubin1963,Lions1969}.
This result has been extended, improved, and reviewed several times; see
\cite{Simon1987,Andrei2011,Jungel2012,Gllou2012,Jungel2014,Elliott2021} and
references therein. In particular, we note that in practice it is very difficult
to obtain estimates for the discrete time derivative. Therefore, there have been
major efforts in order to replace bounds on the derivative by some form of
equicontinuity, or uniform modulus continuity in Bochner-type norms.
\end{remark}

Let us now state the additional set of assumptions we shall impose on $\calA$.
 These are lower bounds and growth conditions for $\langle
\mathcal{A}(w),w\rangle$. More precisely, we will assume that:
\begin{enumerate}[$\bullet$]
  \item $p$--coercivity: There exist $p>1$ and $C_1 >0$ such
that\footnote{This can be generalized, with $C_2>0$, to the case
  \[
    \langle \mathcal{A}(w),w \rangle \geq C_1 \| w \|^p - C_2 | w |^2,
  \]
  but this will inevitably lead to conditional stability in our schemes, or to
the need of Gr\"onwall inequalities.}
  \begin{equation}\label{pcoercivity}
    \langle \mathcal{A}(w),w \rangle \geq C_1 \| w \|^p, \quad \forall w \in
\mathbb{V}.
  \end{equation}

  \item $q$--growth: There exist $q \geq p$ and an increasing function
$C_3: \mathbb{R} \to \mathbb{R}$ for which\footnote{This can also be generalized
to have lower order terms.}
  \begin{equation}\label{rgrowth}
    \|\mathcal{A}(w)\|_{*} \leq C_3(|w|) \|w\|^{p/q'}, \quad \forall w \in
\mathbb{V}.
  \end{equation}
\end{enumerate}

We highlight that conditions \eqref{pcoercivity} and \eqref{rgrowth} indeed do
appear in a large class of problems of mathematical and technical interest; see
Appendix \ref{AppExampleProb}.

Notice that our assumptions allow us to obtain a priori estimates on the
solution and its derivative.

\begin{proposition}[a priori estimates]
Assume that $f \in L^{p'}(0,\tf;\mathbb{V}^*)$. If $\calA$ satisfies
\eqref{pcoercivity}, then we have that the solution to \eqref{eq:TheODE}
satisfies
\begin{equation}
\label{eq:ProbIsBochnerStable}
  \frac12 \|u \|_{L^\infty(0,\tf;\polH)}^2 + \frac{C_1}{p'}\| u
\|_{L^p(0,\tf;\polV)}^p \leq \frac12 |u_0 |^2 + \frac1{p'C_1^{1/(p-1)}} \| f
\|_{L^{p'}(0,\tf;\polV')}^{p'}.
\end{equation}
If, in addition, $\calA$ satisfies \eqref{rgrowth} we also have
\begin{equation}
\label{eq:DerivEstimateNegativeNorm}
  \left\| \frac{\diff u}{\diff t} \right\|_{L^{q'}(0,\tf;\polV^*)}^{q'} \lesssim
|u_0|^2 + \| f \|_{L^{p'}(0,\tf;\polV')}^{p'}.
\end{equation}
\end{proposition}
\begin{proof}
We recall that, in this setting, equation \eqref{eq:TheODE} must be understood
in $\polV^*$ for \emph{a.e.}~$t \in(0,T)$. We apply said functional to $u(t)$
and obtain
\[
  \frac12 \frac{\diff}{\diff t} | u (t) |^2 + C_1 \| u (t) \|^p \leq
\left\langle \frac{\diff u(t)}{\diff t},u(t) \right\rangle + \langle
\mathcal{A}(u(t)),u(t) \rangle = \langle f(t), u(t) \rangle \leq
\frac1{p'\varepsilon^{p'}} \| f(t) \|_{*}^{p'} + \frac{\varepsilon^p}p \| u(t)
\|^p,
\]
where the lower bound is obtained using the coercivity condition
\eqref{pcoercivity},
and the upper bound is obtained using Young's inequality. Integrating we can
conclude that
\[
  \frac12 \| u \|_{L^\infty(0,\tf;\mathbb{H})}^2 + \left(C_1 -
\frac{\varepsilon^p}p \right) \| u \|_{L^p(0,\tf;\mathbb{V})}^p \leq \frac12
|u_0|^2 + \frac1{p'\varepsilon^{p'}} \| f \|_{L^{p'}(0,\tf;\mathbb{V}^*)}^{p'}.
\]
A suitable choice of $\varepsilon>0$ then shows \eqref{eq:ProbIsBochnerStable}.

On the other hand, by definition, we have
\begin{align*}
  \left\| \frac{\diff u(t)}{\diff t} \right\|_* &= \sup_{0 \neq w \in
\mathbb{V}} \frac{ \langle \tfrac{\diff u(t)}{\diff t},w\rangle}{ \| w \|} =
\sup_{0 \neq w \in \mathbb{V}} \frac{ \langle f(t) -
\mathcal{A}(u(t)),w\rangle}{ \| w \|} \\
  &\leq \| f(t) \|_* + \| \mathcal{A}(u(t)) \|_* \leq \| f(t) \|_* + C_3(|u(t)|)
\| u(t) \|^{p/q'}.
\end{align*}
where in the last step we used the $q$--growth condition. Notice now that,
the uniform $L^\infty(0,\tf;\polH)$ estimate on $u$, and the fact that $C_3$ is
increasing, imply that
\[
  \sup_{t\in[0,\tf]} C_3(|u(t)|) \leq \bar{C}_3 < \infty.
\]
Thus, raising to power $q'$ the previous estimate and integrating we get
\[
  \left\| \frac{\diff u}{\diff t} \right\|_{L^{q'}(0,\tf;\mathbb{V}^*)}^{q'}
\lesssim \| f \|_{L^{q'}(0,\tf;\mathbb{V}^*)}^{q'} + \| u(t)
\|_{L^p(0,\tf;\polV)}^p.
\]
Since $q'\leq p'$ we have that $L^{p'} \hookrightarrow L^{q'}$, and the desired
derivative estimate follows from \eqref{eq:ProbIsBochnerStable}.
\end{proof}

Let us now turn to DIRK schemes. In light of the a priori estimates presented
above we now introduce our strongest notion of stability, namely, discrete
Bochner stability.

\begin{definition}[discrete Bochner stability]\label{def:discBochnerStable}
We will say that the $s$-stage DIRK scheme with tableau \eqref{eq:TheDIRK} is
Bochner stable when applied to the problem \eqref{eq:TheODE} if there are
strictly positive $C_A,C_f$, $\{\omega_i\}_{i=1}^s$, $\{\gamma_i\}_{i=1}^s$,
$p>1$, and $q \geq p$ such that, for all $\frakN \in \polN$, any partition
$\calP_\frakN$, and all $n \in \{0, \ldots, \frakN -1 \}$, we have
\begin{equation}\label{eq:discrBochnerStable}
  \frac12 |u_{n+1}|^2 + C_A \dt_n \sum_{i=1}^s \omega_i \| v_{n,i} \|^p \leq
\frac12 |u_n|^2 + C_f \dt_n \sum_{i=1}^s \gamma_i \| f_{n,i} \|_*^{q'}.
\end{equation}
\end{definition}

\begin{remark}[notation]
Notice that, since $\gamma_i >0$ for all $i = 1, \ldots, s$, the linear
functional
\[
  \varphi \mapsto \frac1{\sum_{i=1}^s \gamma_i} \sum_{i=1}^s \gamma_i
\varphi(c_i), \qquad \forall \varphi \in \mathcal{C}([0,1])
\]
defines a quadrature formula that is exact for at least constant functions.
For this reason, and to alleviate notation, we shall define, for $r \in
(1,\infty)$,
\[
  \| f \|_{L_{\mathcal{P}_N}^r(0,\tf;\polV^*)} = \left( \sum_{n=1}^\frakN \dt_n
\sum_{i=1}^s \gamma_i \| f(t_n + c_i \tau_n) \|_*^{r}
  \right)^{1/r}.
\]
\end{remark}

It turns out that satisfaction of a dissipative discrete energy-balance such as
 \eqref{AbstractDiscBal} is a rather strong foundation for stability in the
context of practical applications. In fact, under our additional assumptions on
$\calA$, Definition~\ref{def:DiscrDissEnergyBalance} implies
\eqref{eq:discrBochnerStable}.

\begin{proposition}[discrete Bochner stability]\label{prop:DiscBochnerIsImplied}
Assume that the mapping $\calA$ satisfies the coercivity assumption
\eqref{pcoercivity}. If a DIRK scheme satisfies the dissipative discrete energy
balance \eqref{AbstractDiscBal}, then it is discretely Bochner stable in the
sense of Definition~\ref{def:discBochnerStable}.
\end{proposition}
\begin{proof}
We note that \eqref{AbstractDiscBal} can be rewritten as
\[
  \frac12 |u_{n+1}|^2 + \dt_n \sum_{i=1}^s
\nu_i \langle \calA(v_{n,i}), v_{n,i} \rangle
\leq
\frac12 |u_n|^2 + \dt_n
\sum_{i=1}^s \nu_i \| f_{n,i} \|_* \|v_{n,i}\|.
\]
Using the $p$--coercivity condition on $\calA$ the left hand side of this
inequality can be bounded from below. An application of Young's inequality on
the right hand side, and the fact that $q'\geq p'$, then leads to
\eqref{eq:discrBochnerStable}.
\end{proof}

The following result is further evidence of the relevance of remarkably stable
schemes.

\begin{corollary}[remarkable stability]
If a DIRK scheme is remarkably stable in the sense of
Definitions~\ref{RemarkableDefTwoStage} or \ref{def:RemarkableThreeStage}, then
it is discretely Bochner stable in the sense of
Definition~\ref{def:discBochnerStable}.
\end{corollary}
\begin{proof}
Owing to Proposition~\ref{prop:DiscBochnerIsImplied} it suffices to recall that
remarkably stable schemes have a dissipative discrete energy law of the form
\eqref{AbstractDiscBal}.
\end{proof}

Examples of two- and three-stage schemes that are remarkably stable are
presented in Sections~\ref{sub:ExampleTwoStageRemarkable} and
\ref{sub:ExampleThreeStageRemarkable}. In the remainder of this section we
further explore properties of discretely Bochner stable schemes.

Notice, first of all, that \eqref{eq:discrBochnerStable} can be thought of as a
discrete and local, in time, version of \eqref{eq:ProbIsBochnerStable}, as the
following result shows.

\begin{proposition}[global stability]
\label{prop:GlobDIRKstab}
Assume that \eqref{eq:DIRKappliedtoPDE} satisfies
\eqref{eq:discrBochnerStable}.
Then, for all $\frakN \in \polN$, and any partition $\calP_\frakN$ we have
\[
  \frac12 \max_{n=0}^\frakN | u_n |^2 + C_A \sum_{n=1}^\frakN \tau_n \sum_{i=1}^{s} \omega_i \| v_{n,i} \|^p \leq
  \frac12 | u_0 |^2 + C_f \| f \|_{L^{q'}_{\mathcal{P}_N}(0,\tf,\polV^*)}^{q'}.
\]
\end{proposition}
\begin{proof}
It suffices to add \eqref{eq:discrBochnerStable} over $n$.
\end{proof}

As a final, important, property of discrete Bochner stable schemes we now show
that, under the assumption that \eqref{eq:discrBochnerStable} holds, an estimate
on the discrete time derivative of the solution, in the spirit of
\eqref{eq:DerivEstimateNegativeNorm}, can be obtained. We begin with a uniform
bound on the stages of the form
\begin{equation}
\label{eq:frakC3}
  \mathfrak{C}_3 = \sup_{\frakN \in \polN} \sup_{\calP_\frakN} \sup_{n \in \{0, \ldots, \frakN-1\}} \max_{i=1}^s C_3 \left( |v_{n,i}| \right) < \infty.
\end{equation}

\begin{lemma}[bound on stages]
\label{lem:stagebound}
Assume that the scheme \eqref{eq:DIRKappliedtoPDE} satisfies \eqref{eq:discrBochnerStable}. Then, there is a constant $C>0$ such that for every $\frakN \in \polN$, and every partition $\calP_\frakN$ we have.
\[
  \max_{n=0}^{\frakN-1} \max_{i=1}^s | v_{n,i} | \leq C.
\]
The constant $C$ may depend on $s$, the entries of the tableau
\eqref{eq:TheDIRK}, and natural norms on the data, i.e., $|u_0|$ and $\| f
\|_{L^{p'}(0,\tf;\polV^*)}$. As a consequence, we have that \eqref{eq:frakC3}
holds.
\end{lemma}
\begin{proof}
The idea of the proof is to exploit the fact that $\bv{A}$ is lower triangular
and with positive diagonal entries. For simplicity we present the proof for the
case $s=2$, but the reader may easily verify that the procedure extends to
arbitrary $s$.

Let $n \in \{0, \ldots,\frakN-1\}$ and consider the first stage,
\[
  v_{n,1} - u_n + a_{11}\dt_n \calA(v_{n,1}) = a_{11}\dt_n f_{n,1}.
\]
Testing this identity with $v_{n,1}$ we arrive at
\[
  \frac12 \left( |v_{n,1}|^2 + |v_{n,1}-u_n|^2 - |u_n|^2 \right) + C_1a_{11}\dt_n \| v_{n,1} \|^p \leq a_{11}\dt_n \left( \frac{\| f_{n,1} \|_*^{p'}}{p'C_1^{1/(p-1)}} + \frac{ C_1\| v_{n,1} \|^p}{p} \right),
\]
where we used the $p$--coercivity condition \eqref{pcoercivity} and Young's
inequality. Rearranging we have obtained that
\begin{equation}
\label{eq:frakC3forv1}
  \begin{aligned}
    \frac12 |v_{n,1}|^2 + \frac12 |v_{n,1}-u_n|^2 + \frac{C_1a_{11}\dt_n}{p'} \| v_{n,1} \|^p &\leq \frac12 |u_n|^2 + \frac{a_{11}\dt_n}{p'C_1^{1/(p-1)}}\| f_{n,1} \|_*^{p'} \\
      &\leq \frac12 \max_{n=0}^\frakN |u_n|^2 + \kappa_1 \| f \|_{L_{\mathcal{P}_N}^{p'}(0,\tf;\polV^*)}^{p'}
  \end{aligned}
\end{equation}
where the constant $\kappa_1$ depends only on $a_{11}$, $p$, and $C_1$.
In conclusion, for every $n$, $|v_{n,1}|$ is uniformly bounded only in terms of
data.

With the bound on the first stage at hand we proceed to bound the second stage,
which we write as
\[
  v_{n,2} - u_n + a_{22}\dt_n \calA(v_{n,2}) = a_{22 } \dt_n f_{n,2} + a_{21}\dt_n \scrF_{n,1}.
\]
Testing with $v_{n,2}$ yields
\[
  \frac12\left[ |v_{n,2}|^2 +|v_{n,2} - u_n|^2 - |u_n|^2 \right] + C_1 a_{22}\dt_n \| v_{n,2} \|^p \leq
  a_{22}\dt_n \left(
    \frac{ \| f_{n,2} \|_*^{p'} }{p'C_1^{1/(p-1)}}  + \frac{C_1 \| v_{n,2} \|^p}p
  \right)
  + a_{21}\dt_n \langle \scrF_{n,1}, v_{n,2} \rangle.
\]
We now multiply the equation that defines the first stage by $v_{n,2}$ to get
\[
  \frac{a_{21}}{a_{11}} (v_{n,1} - u_n, v_{n,2}) = a_{21}\tau_n \langle \scrF_{n,1}, v_{n,2} \rangle.
\]
In summary, we have obtained that
\begin{align*}
  \frac12\left[ |v_{n,2}|^2 +|v_{n,2}-u_n|^2  \right] + \frac{a_{22}C_1 \dt_n}{p'} \| v_{n,2} \|^p &\leq
    \frac12 |u_n|^2 + \frac{a_{22}\dt_n}{p'C_1^{1/(p-1)}}\| f_{n,2} \|_*^{p'} + \frac{a_{21}}{a_{11}} (v_{n,1}- u_n, v_{n,2}) \\
    &\leq \frac12 \max_{n=0}^\frakN |u_n|^2 +
    \kappa_2 \| f \|_{L_{\mathcal{P}_N}^{p'}(0,\tf;\polV^*)}^{p'}
    + \frac14 | v_{n,2}|^2 \\
      &+\left(\frac{a_{21}}{a_{11}}\right)^2 |v_{n,1} - u_n|^2,
\end{align*}
where $\kappa_2$ is a constant that depends only on $\bf{A}$, $p$, and $C_1$.
The previously obtained bound \eqref{eq:frakC3forv1} then shows that
\begin{align*}
  \frac14\left[ |v_{n,2}|^2 +|v_{n,2}-u_n|^2  \right] + \frac{a_{22}C_1 \dt_n}{p'} \| v_{n,2} \|^p &\leq
    \left[\frac12 + \left(\frac{a_{21}}{a_{11}}\right)^2 \right] \max_{n=0}^\frakN |u_n|^2 \\
    &+ \left[\kappa_2 + 2\left(\frac{a_{21}}{a_{11}}\right)^2\kappa_1 \right] \| f \|_{L_{\mathcal{P}_N}^{p'}(0,\tf;\polV^*)}^{p'}
,
\end{align*}
so that, for all $n$, $|v_{n,2}|$ is also uniformly bounded in terms of data.
\end{proof}

\begin{proposition}[derivative estimate]
Assume that the scheme \eqref{eq:DIRKappliedtoPDE} is at least first order
accurate, and that it satisfies \eqref{eq:discrBochnerStable} and
\eqref{eq:frakC3}. Then, for any $\frakN \in \polN$, and any partition
$\calP_\frakN$, the solution to \eqref{eq:DIRKappliedtoPDE} satisfies
\[
  \left[ \mu\sum_{n=0}^{\frakN-1} \tau_n \left\| \frac{u_{n+1} - u_n}{\tau_n} \right\|_*^{q'} \right]^{1/q'}\leq
  \left( \mu \tf \right)^{\frac{p'-q'}{p'q'}}
  \| f \|_{L_{\mathcal{P}_N}^{p'}(0,\tf;\polV^*)}
  +
  \mathfrak{C}_3 \left[ \sum_{n=1}^\frakN \tau_n \sum_{i=1}^s \omega_i \|v_{n,i}\|^{p} \right]^{1/q'},
\]
where
\[
  \mu = \min\left\{ \min_{i=1}^s \gamma_i, \min_{i=1}^s \omega_i \right\},
\]
and $\mathfrak{C}_3>0$ is defined in \eqref{eq:frakC3}.
\end{proposition}
\begin{proof}
We begin by recalling that $b_i\geq 0$. Moreover, since the scheme is at least
first order accurate, \eqref{eq:OrderOneConditions} holds and, consequently, for
all $r \in [1,\infty)$
\begin{equation}
\label{eq:DerEstTrivialButImportant}
  \left( \sum_{i=1}^s b_i^r \right)^{1/r} \leq \sum_{i=1}^s b_i = 1.
\end{equation}

Notice now that the last equation in \eqref{eq:DIRKappliedtoPDE} implies
\[
  \left\| \frac{u_{n+1} - u_n}{\tau_n} \right\|_* \leq \sum_{i=1}^s b_i \| f(t_n + c_i \tau_n) \|_* + \mathfrak{C}_3 \sum_{i=1}^s b_i \| v_{n,i} \|^{p/q'},
\]
where we used the $q$--growth condition \eqref{rgrowth}. Raise this inequality
to power $q'$, multiply the result by $\mu \tau_n$, and add over $n$ to obtain
\begin{equation}
\label{eq:DerEstimateBeforeHolders}
  \begin{aligned}
    \left[ \mu \sum_{n=0}^{\frakN-1} \tau_n \left\| \frac{u_{n+1} - u_n}{\tau_n} \right\|_*^{q'} \right]^{1/q'} &\leq \left[ \sum_{n=0}^{\frakN-1} \left| A_n + B_n \right|^{q'} \right]^{1/q'}
    &\leq \left[ \sum_{n=0}^{\frakN-1} |A_n|^{q'} \right]^{1/q'} + \left[ \sum_{n=0}^{\frakN-1} |B_n|^{q'} \right]^{1/q'},
  \end{aligned}
\end{equation}
where we denoted
\[
  A_n = \mu^{1/q'}\tau_n^{1/q'} \sum_{i=1}^s b_i \|f(t_n+ c_i \tau_n)\|_*, \qquad
  B_n = \mathfrak{C}_3^{1/q'}\mu^{1/q'}\tau_n^{1/q'} \sum_{i=1}^s b_i \|v_{n,i}\|^{p/q'}.
\]

We now estimate each term on the right hand side of
\eqref{eq:DerEstimateBeforeHolders} separately. Using repeatedly H\"older's
inequality we observe that
\begin{align*}
  \left[ \sum_{n=0}^{\frakN-1} |A_n|^{q'} \right]^{1/q'}
  &= \left[\sum_{n=0}^{\frakN-1} \mu\tau_n \left| \sum_{i=1}^s b_i \|f(t_n+ c_i \tau_n)\|_* \right|^{q'} \right]^{1/q'}
  \\
  &\leq \left[\sum_{n=0}^{\frakN-1} \mu\tau_n\right]^{\frac1{q'}\frac1{(p'/q')'}}
    \left[ \sum_{n=0}^{\frakN-1} \mu\tau_n \left| \sum_{i=1}^s b_i \|f(t_n+ c_i \tau_n)\|_* \right|^{p'}\right]^{\frac1{q'}\frac1{p'/q'}}
  \\
  &\leq \left( \mu \tf \right)^{\frac{p'-q'}{p'q'}}\left[ \sum_{n=0}^{\frakN-1} \tau_n \sum_{i=1}^s \gamma_i \|f(t_n+ c_i \tau_n)\|_*^{p'} \right]^{1/p'} =
  \left( \mu \tf \right)^{\frac{p'-q'}{p'q'}} \| f \|_{L_{\mathcal{P}_N}^{p'}(0,\tf;\polV^*)},
\end{align*}
where we used \eqref{eq:DerEstTrivialButImportant}, and that $\mu \leq
\gamma_i$.
Similarly,
\begin{align*}
  \left[ \sum_{n=0}^{\frakN-1} |B_n|^{q'} \right]^{1/q'} &= \mathfrak{C}_3\left[\sum_{n=0}^{\frakN-1} \mu\tau_n \left| \sum_{i=1}^s b_i \|v_{n,i}\|^{p/q'} \right|^{q'} \right]^{1/q'} \\
  &\leq \mathfrak{C}_3\left[ \sum_{n=0}^{\frakN-1} \tau_n \sum_{i=1}^s \omega_i \|v_{n,i}\|^{p} \right]^{1/q'},
\end{align*}
where we again used \eqref{eq:DerEstTrivialButImportant} and $\mu \leq \omega_i$. The result follows.
\end{proof}

Under further structural assumptions on $\calA$, like hemicontinuity
\cite[Definition 2.3]{Roubicek2013}, the previous results together with a
standard exercise in compactness (\cf~Lemma~\ref{AubinLionsLemma}) allow us to
assert convergence of discretely Bochner stable DIRK schemes under minimal
regularity assumptions. We shall not dwell on this.

\appendix

\section{Example problems accommodating our assumptions}\label{AppExampleProb}

Let us present several examples of problems that our framework can manage.
We will indicate when such problems satisfy our minimal set of assumptions,
i.e., \eqref{positivity}, and when they do satisfy the more stringent
assumptions of Section~\ref{sec:AisCoercive}. For a more comprehensive list and
insight we refer the reader to \cite[Chapter 8]{Roubicek2013}. In all of the
descriptions below $d \geq 1$, and $\Omega \subset \mathbb{R}^d$ is a bounded
domain with, at least, Lipschitz boundary.

\subsection{Nonlinear diffusion equations}\label{sub:NLdiffusion}
Let $K : \Omega \times\mathbb{R} \to \mathbb{R}^{d \times d}$ be bounded,
measurable, and nonnegative definite, i.e.,
\[
  \boldsymbol{\xi}^\intercal K(x,s) \boldsymbol{\xi} \geq 0, \quad \forall x\in \Omega, \quad \forall s \in \mathbb{R}, \quad \forall \boldsymbol{\xi} \in \mathbb{R}^d.
\]
The problem
\[
  \begin{dcases}
    \partial_t u(x,t) -\diver{} \left( K(x,u(x,t)) \nabla u(x,t) \right) = f(x,t), & (x,t) \in \Omega \times (0,T), \\
    u(x,t) = 0, & (x,t) \in \partial\Omega \times (0,T),\\
    u(x,0) = u_0(x), & x \in \Omega,
  \end{dcases}
\]
can be cast into \eqref{eq:TheODE} by setting $\polV = H^1_0(\Omega)$, $\polH = L^2(\Omega)$ and
\[
  \langle \calA(v), w \rangle = \int_\Omega \nabla w(x)^\intercal K(x,v(x)) \nabla v(x)  \diff x.
\]
Clearly this operator satisfies \eqref{positivity}.

If, in addition, we assume that $K$ is uniformly bounded, and uniformly
positive definite, that is, there is $K_0>0$ such that
\[
  K_0^{-1} |\boldsymbol{\xi}|^2 \leq \boldsymbol{\xi}^\intercal K(x,s) \boldsymbol{\xi} \leq K_0 |\boldsymbol{\xi}|^2, \quad \forall x\in \bar\Omega, \quad \forall s \in \mathbb{R}, \quad \forall \boldsymbol{\xi} \in \mathbb{R}^d,
\]
then with $p = q = 2$, $C_1 = K_0^{-1}$ and $C_3 = K_0$ this problem satisfies
\eqref{pcoercivity} and \eqref{rgrowth}.

\subsection{Nonlinear diffusion reaction problems}
The previous example can be slightly generalized to
\[
  \partial_t u(x,t) -\diver{}\left( K(x,u(x,t)) \nabla u(x,t) \right) + \gamma(x,u(x,t)) = f(x,t),
\]
where $\gamma: \Omega \times \mathbb{R} \to \mathbb{R}$ is nonnegative for a
nonnegative second argument, and it has sufficiently mild growth. For instance,
if $\gamma = |w|^2 w$, we see that
\[
  \gamma(w) w = |w|^4 \geq 0,
\]
so we get positivity, i.e., \eqref{positivity}. If, in addition, $d=2$ we recall
 that for $v,w \in H^1_0(\Omega)$
\[
  \int_\Omega |w|^3 v \diff x \leq \| w \|_{L^6(\Omega)}^3 \| v \|_{L^2(\Omega)} \leq C \| w \|_{L^6(\Omega)}^3 \| \nabla v \|_{L^2(\Omega;\mathbb{R}^2)},
\]
where we used Poincar\'e inequality. Now, the Gagliardo--Nirenberg interpolation
 inequality \cite[Theorem 1.24]{Roubicek2013} implies that
\[
  \| w \|_{L^6(\Omega)}^3 \leq C\left[ \| w \|_{L^2(\Omega)}^{5/9} \| \nabla w \|_{L^2(\Omega;\mathbb{R}^2)}^{4/9} \right]^3 =C \| w \|_{L^2(\Omega)}^{5/3} \| \nabla w \|_{L^2(\Omega;\mathbb{R}^2)}^{4/3}.
\]
Consequently, our problem fits into the framework of
Section~\ref{sec:AisCoercive} with $p = 2$, $q=3$, $C_1 = K_0^{-1}$, and $C_3 =
1+ \| u \|_{L^2(\Omega)}^{5/3}$. To see this, it suffices to realize that, since
$q'=\tfrac32$
\[
  \frac43 = \frac2{3/2}.
\]

\subsection{Parabolic quasilinear equations}
One further generalization that nonlinear diffusions allow is the following.
Let $G:\Omega \times \mathbb{R}^d \to \mathbb{R}$ be convex in its second
argument and $\bv{F} = D_2 G$ its derivative with respect to its second
argument. Assume that these functions satisfy classical conditions of the form
\[
  G(x,\boldsymbol{\xi}) \geq \alpha_1 |\boldsymbol{\xi}|^p, \qquad |\bv{F}(x,\boldsymbol{\xi})| \leq \alpha_3 |\boldsymbol{\xi}|^{p-1}, \quad \forall x\in\Omega, \ \boldsymbol{\xi} \in \mathbb{R}^d,
\]
with $p > \max\{1, \tfrac{2d}{d+2} \}$. The equation
\[
  \partial_t u(x,t) - \diver{} \bv{F}(x,\nabla u(x,t)) = f(x,t),
\]
supplemented with suitable initial and boundary conditions, can be cast into
the framework of Section~\ref{sec:AisCoercive} with $p=q$, $\polH =
L^2(\Omega)$, and $\polV = W^{1,p}_0(\Omega)$. Clearly, $C_1= \alpha_1$ and $C_3
= \alpha_3$.

A classical example of this scenario is the parabolic $p$--Laplacian problem
\[
  \partial_t u(x,t) - \diver{} \left(|\nabla u(x,t)|^{p-2} \nabla u(x,t) \right) = f(x,t).
\]
To see this, it suffices to set $G(x,\boldsymbol{\xi}) = \tfrac1p |\boldsymbol{\xi}|^p$.

\subsection{The Navier-Stokes equations} \label{sub:NSE}
The well known Navier-Stokes equations read
\[
  \partial_t \bv{u}(x,t) + \diver{}\left[ \bv{u}(x,t)\otimes \bv{u}(x,t) \right]- \nu \Delta \bv{u}(x,t) + \nabla \pi(x,t)= \bv{f}(x,t), \quad \diver{} \bv{u}(x,t) = 0,
\]
and are supplemented with suitable initial and boundary conditions. Here $\nu>0$
 is the viscosity. To see how this problem fits the framework of
Section~\ref{sec:AisCoercive} we set, for definiteness, $d=3$ and
\begin{align*}
  \polH &= \left\{ \bv{v} \in L^2(\Omega;\mathbb{R}^3) : \diver{} \bv{v} = 0, \bv{v} \cdot \bv{n}_{|\partial\Omega} = 0 \right\}, \\
  \polV &= H^1_0(\Omega;\mathbb{R}^3) \cap \polH.
\end{align*}
The operator $\calA$ is defined as
\[
  \langle \calA(\bv{v}), \bv{w} \rangle = \nu \int_\Omega \nabla \bv{v} : \nabla \bv{w} \diff x - \int_\Omega (\bv{v} \otimes \bv{v} ):\nabla \bv{w} \diff x.
\]
Owing to the skew symmetry of the convective term (over divergence free fields)
we have
\[
  \nu \| \nabla \bv{w} \|_{L^2(\Omega;\mathbb{R}^{3 \times 3})}^2 \leq \langle \calA( \bv{w} ), \bv{w} \rangle,
\]
so that, clearly, $C_1 = \nu$ and $p=2$.

Consider now
\[
  \left|\int_\Omega (\bv{v} \otimes \bv{v} ):\nabla \bv{w} \diff x\right| \leq \| \bv{v} \|_{L^4(\Omega;\mathbb{R}^3)}^2 \| \nabla \bv{w} \|_{L^2(\Omega;\mathbb{R}^{3 \times 3})} \leq
  \left[\| \bv{v} \|_{L^2(\Omega;\mathbb{R}^3)}^{1/4} \| \nabla \bv{v} \|_{L^2(\Omega;\mathbb{R}^{3 \times 3})}^{3/4}\right]^2 \| \nabla \bv{w} \|_{L^2(\Omega;\mathbb{R}^{3 \times 3})},
\]
where we, again, used the Gagliardo--Nirenberg interpolation inequality.
This shows that
\[
  C_3 = 1+\| \bv{v} \|_{L^2(\Omega;\mathbb{R}^3)}^{1/2},
\qquad \frac32 = \frac{p}{q'} \ \implies \ q' = \frac43.
\]

\subsection{Hamiltonian problems}\label{sub:Hamilton}
The operator $\mathcal{A}$ is linear and induces a skew symmetric bilinear form
on $\mathbb{V}$, i.e.,
\begin{align}\label{SkewSymm}
  \langle \mathcal{A}(u),v \rangle = - \langle \mathcal{A}(v),u \rangle,
\quad \forall u,v \in \mathbb{V}.
\end{align}
This is the prototypical case of Hamiltonian problems such as Maxwell's
equations in free space.

\subsection{GENERIC systems}\label{GenericSec}
The operator $\mathcal{A}$ is a combination of the cases in
Sections~\ref{sub:NLdiffusion} and \ref{sub:Hamilton}, that is, a combination of
a dissipative and a Hamiltonian parts. For instance we could consider, for
$\epsilon \geq 0$, an operator of the form $\mathcal{A}(w) = \mathcal{S}(w) +
\epsilon \mathcal{D}(w)$ where, for all $v,w \in \mathbb{V}$,
we have
\begin{align*}
\langle \mathcal{S}(w), v \rangle = - \langle w, \mathcal{S}(v) \rangle,
\qquad \langle \mathcal{D} w,w \rangle \geq 0.
\end{align*}

This type of PDE problems are usually called GENERIC
\cite{Ottinger1997,Mielke2011,Egger2019,Eldred2020}. For instance, the linear
wave equation with damping is a GENERIC system. Similarly, incompressible
Navier-Stokes equations could be understood as the sum of a dissipative system
(i.e., the bilinear form associated to viscous effects) and a nonlinear
Hamiltonian system (the skew symmetric trilinear form associated to convective
terms), see for instance \cite{Temam1979,MarionTemam1998}.


\section{Some properties of non-remarkable schemes}\label{AlexanderAppendix}

We start with a rather trivial observation.

\begin{remark}[non-remarkable schemes and skew-symmetric problems]\label{RemarkUnsigned}
Consider \eqref{eq:TheODE} with $f \equiv 0$ and $\mathcal{A}$ a skew-symmetric operator, i.e.,
\[
  \langle\mathcal{A}(u),v\rangle = - \langle\mathcal{A}(v),u\rangle, \quad \forall u,v \in \polV.
\]
In other words, we consider purely autonomous dynamics. A non-remarkable
two-stage DIRK scheme, meaning a scheme that does not satisfy the
properties described in Definition \ref{RemarkableDefTwoStage}, will satisfy the
following discrete energy-balance
\begin{align}\label{SkewBal22}
\frac{1}{2}|u_{n+1}|^2 + \mathcal{Q}(u_n,v_{n,1},v_{n,2}) =
\frac{1}{2}|u_{n}|^2 \, ,
\end{align}
where the quadratic form $\mathcal{Q}$, introduced in \eqref{QformTwoStage}, is unsigned.
Similarly, a non-remarkable three-stage DIRK scheme, meaning a scheme that does
not satisfy the properties described in Definition~\ref{def:RemarkableThreeStage}, will satisfy the discrete energy-balance
\begin{align}\label{SkewBal33}
\frac{1}{2}|u_{n+1}|^2 + \mathcal{Q}(u_n,v_{n,1},v_{n,2}, v_{n,3}) =
\frac{1}{2}|u_{n}|^2 \, ,
\end{align}
where $\mathcal{Q}$, defined in \eqref{Qform33}, is unsigned.
\end{remark}

Identities \eqref{SkewBal22} and \eqref{SkewBal33} are an immediate
consequence of \eqref{GeneralEnergyBal} and \eqref{EnergyID33poly}
respectively. They follow from the fact that if $\mathcal{A}$ is
skew-symmetric, then we have that $\langle\mathcal{A}(v_{n,i}),v_{n,i}\rangle = 0$ for
each stage $i \in \{1, \ldots, s \}$. Remark \ref{RemarkUnsigned} tells us that
non-remarkable schemes cannot be guaranteed to be stable when applied to problems
of skew-symmetric nature. The same holds true for GENERIC-like PDEs, see
Section \ref{GenericSec}, with $\epsilon > 0$ sufficiently small.

While simple, this is an important observation. Most nonlinear problems,
either locally in time, or through linearization, can be thought as having a
symmetric and skew-symmetric parts. The symmetric part is usually positive and
related to dissipative behavior. The skew-symmetric part describes conserved
quantities and/or wave-like nature of the problem. In this regard, Remark
\ref{RemarkUnsigned} tells us that unconditional stability cannot be expected
when using non-remarkable schemes for PDEs strongly dominated by their
skew-symmetric part. In this very specific context of non-remarkably stable
schemes and problems with skew-symmetric operator it is pointless to attempt to
develop any theory regarding convergence or error estimates; or to engage in any
discussion related to order-reduction; or to compare its performance to other
schemes. This is because, to begin with, the scheme cannot be proven to be
stable. For many problems, say for instance the linear acoustic wave equations
in first-order form, Remark \ref{RemarkUnsigned} should be the final argument
against the use of non-remarkably stable schemes.

We notice, in particular, that the very popular Alexander's DIRK22 and DIRK33
L-stable schemes, described by tableaus \eqref{AlexDIRK22tableu} and
\eqref{DIRK33Lstable} respectively, are not remarkably stable. This severely
limits their applicability to the solution of evolutionary PDEs. For the sake of
completeness we present an optimal proof of stability for Alexander's DIRK22
scheme in the context of linear, self-adjoint, and positive operators.

\begin{proposition}[energy identity I]\label{PropAlexanderI}
Consider \eqref{eq:TheODE} with $f \equiv0$. Assume that the operator
$\mathcal{A}$ is linear; coercive/positive-definite, i.e., \eqref{pcoercivity}
holds with $p=2$; and  symmetric, that is,
\[
  \langle \calA(v), w \rangle = \langle \calA(w), v \rangle, \quad \forall v,w \in \polV.
\]
Identity \eqref{eq:protoenergy} for Alexander's DIRK22 scheme takes the
following specific form
\begin{align}\label{Alex22EnergyId}
\begin{split}
&\tfrac{1}{2}|u_{n+1}|^2 - \tfrac{1}{2}|u_n|^2
+ \tfrac{1}{2}|v_{n,1} - u_n|^2
+ \tfrac{1}{2}|u_{n+1} - v_{n,1}|^2 \\
& \ \ \ + \dt_n \gamma \langle \mathcal{A} (v_{n,1}), v_{n,1}\rangle
+ \dt_n \gamma \langle \mathcal{A} (u_{n+1}), u_{n+1}\rangle
= - \dt_n (1-2\gamma) \langle \mathcal{A} (v_{n,1}), u_{n+1}\rangle
\end{split}
\end{align}
with $\gamma$ as defined in \eqref{AlexDIRK22tableu}.
\end{proposition}
\begin{proof}
One only needs to use the values of the tableau, which are given in
\eqref{AlexDIRK22tableu}.
\end{proof}

As usual the problem lies with the unsigned off-diagonal term $\langle
\mathcal{A} (v_{n,1}), u_{n+1}\rangle$ in the right-hand side of
\eqref{Alex22EnergyId}. We may consider absorbing part of it into the
artificial damping terms as described in the following result.

\begin{proposition}[partial damping]\label{PropAlexanderII}
Let $\kappa \in \mathbb{R}$ be any real number. In the setting of Proposition~\ref{PropAlexanderI}, the energy
balance \eqref{Alex22EnergyId} can be rewritten as
\begin{align}
\begin{split}\label{Alex22EnergyIdtwo}
&\tfrac{1}{2}|u_{n+1}|^2 - \tfrac{1}{2}|u_n|^2
+ \mathcal{Q}_{\kappa}(u_n, v_{n,1}, u_{n+1})
+ \dt_n (\gamma + \kappa) \langle \mathcal{A} (v_{n,1}), v_{n,1}\rangle \\
&\ \ \ + \dt_n \gamma \langle \mathcal{A} (u_{n+1}), u_{n+1}\rangle =
- \dt_n (1-2\gamma - \kappa) \langle \mathcal{A} (v_{n,1}), u_{n+1}\rangle \, ,
\end{split}
\end{align}
where $\mathcal{Q}_{\kappa}(u_n, v_{n,1}, u_{n+1})$ is a quadratic form,
depending on the free parameter $\kappa$, defined as
\begin{align}\label{QAlex22}
\mathcal{Q}_{\kappa}(u_n, v_{n,1}, u_{n+1}) = \tfrac{1}{2}|v_{n,1} - u_n|^2
+ \tfrac{1}{2}|u_{n+1} - v_{n,1}|^2
- \tfrac{\kappa}{\gamma} (v_{n,1} - u_n, u_{n+1} - v_{n,1} ) \, .
\end{align}
\end{proposition}
\begin{proof}
One needs to use techniques similar to those advanced in the proof of Lemma
\eqref{PropRewritten22}.
\end{proof}

Given the structure of \eqref{Alex22EnergyIdtwo}--\eqref{QAlex22} we may
want to determine what is the optimal value of $\kappa$ in order to preserve
stability, at the very least when $\mathcal{A}$ is a linear, positive-definite,
symmetric operator.
\begin{enumerate}[$\bullet$]
\item Finding the optimal value of $\kappa$ has two primary
restrictions: we need $\mathcal{Q}_{\kappa}(u_n, v_{n,1}, u_{n+1})$ to
remain non-negative; in addition, we also need to satisfy the property $\gamma + \kappa
> 0$ in order to preserve the a priori bound on $\langle \mathcal{A} (v_{n,1}),
v_{n,1}\rangle$.

\item Setting $\kappa = 1- 2\gamma$ allows us absorb the off-diagonal
term $\dt_n (1-2\gamma - \kappa) \langle \mathcal{A} (v_{n,1}), u_{n+1}\rangle$,
in its entirety, into the quadratic form $\mathcal{Q}_{\kappa}(u_n,v_{n,1},
u_{n+1})$. However, we already know from Section \ref{RemarkAlex22} that is not
feasible. Since Alexander's DIRK22 scheme is not remarkably-stable, the choice
$\kappa = 1- 2\gamma$ will lead to $\mathcal{Q}_\kappa(u_n, v_{n,1},u_{n+1})$ being
unsigned.

\item Some inspection reveals that the largest value of $\kappa$ we
can use, while also retaining non-negativity of $\mathcal{Q}(u_n,
v_{n,1}, u_{n+1})$ and positivity of $\gamma + \kappa$, is $\kappa = \gamma$.
\end{enumerate}

These observations lead to the following result.

\begin{lemma}[a priori energy-estimate]\label{LemmaAlexanderI}
Consider \eqref{eq:TheODE} with $f \equiv0$. Assume that the operator
$\mathcal{A}$ is linear; coercive/positive-definite, i.e., \eqref{pcoercivity}
holds with $p=2$; and  symmetric, that is,
\[
  \langle \calA(v), w \rangle = \langle \calA(w), v \rangle, \quad \forall v,w \in \polV.
\]
Then, the numerical solution using Alexander's DIRK22 scheme satisfies the
following optimal a priori energy estimate:
\begin{align}
\begin{split}
\label{Alex22EnergyOpt}
&\tfrac{1}{2}|u_{n+1}|^2 + \mathcal{Q}_{\gamma}(u_n, v_{n,1}, u_{n+1})
+ \dt_n \big(\tfrac{7 \gamma - 1}{2}\big) \langle \mathcal{A} (v_{n,1}),
v_{n,1}\rangle
+ \dt_n \big(\tfrac{5 \gamma - 1}{2}\big)
\langle \mathcal{A} (u_{n+1}), u_{n+1}\rangle
\leq
\tfrac{1}{2}|u_n|^2
\end{split}
\end{align}
with $\mathcal{Q}_{\gamma}(u_n, v_{n,1}, u_{n+1})$ is defined in \eqref{QAlex22}.
\end{lemma}
\begin{proof} Estimate \eqref{Alex22EnergyOpt} is just a consequence of setting
$\kappa = \gamma$ in \eqref{Alex22EnergyIdtwo} and using Cauchy-Schwarz and
Young's inequality
\begin{align*}
\langle \mathcal{A} (v_{n,1}), u_{n+1}\rangle
\leq
\langle \mathcal{A} (v_{n,1}), v_{n,1}\rangle^{1/2}
\langle \mathcal{A} (u_{n+1}), u_{n+1}\rangle^{1/2}
\leq
\tfrac{1}{2} \langle \mathcal{A} (v_{n,1}), v_{n,1}\rangle
+ \tfrac{1}{2} \langle \mathcal{A} (u_{n+1}), u_{n+1}\rangle \, ,
\end{align*}
for the unsigned off-diagonal term. We claim that \eqref{Alex22EnergyOpt} is
optimal, in the sense that it maximizes the use of artificial damping terms
while also preserving stability of the scheme.
\end{proof}

In conclusion, our assessment is that non-remarkably stable schemes may only be
used either for positive linear problems, or positive nonlinear problems with very
mild growth conditions. They may fail to be stable for problems strongly
dominated by their skew-symmetric component. We assume that the arguments used in
Propositions \ref{PropAlexanderI} and \ref{PropAlexanderII}, and Lemma
\ref{LemmaAlexanderI}, can be extended to the analysis of the Alexander's DIRK33
scheme, but given the observations developed Remark~\ref{RemarkUnsigned}, we
find very little motivations to do so.


\section{Order conditions}\label{AppOrderCond}
It is well known \cite{Hairer1993I,Butcher2008} that the entries of the Butcher
table \eqref{eq:TheDIRK} are bound by the following, necessary, consistency
order conditions:
\begin{enumerate}[$\bullet$]
  \item Order one:
  \begin{equation}
  \label{eq:OrderOneConditions}
    \bv{b}^\intercal \boldsymbol{1} = 1, \qquad \bv{A} \boldsymbol{1} = \bv{c},
  \end{equation}
  where $\boldsymbol{1} =[1] \in \mathbb{R}^s$.

  \item Order two:
  \begin{equation}
  \label{eq:OrderTwoConditions}
    \bv{b}^\intercal \bv{c} = \frac12.
  \end{equation}

  \item Order three:
  \begin{equation}
  \label{eq:OrderThreeConditions}
    \bv{b}^\intercal \bv{A} \bv{c} = \frac16.
  \end{equation}
\end{enumerate}

\section*{Acknowledgements}
The work of AJS is partially supported by NSF grant DMS-2111228. The work of IT
was partially supported by LDRD-express project \#223796 and LDRD-CIS project
\#226834 from Sandia National Laboratories. IT wants to thank John N. Shadid
for his continuous support and encouragement.

\bibliographystyle{plain}

\bibliography{IMEX_DIRK_ref.bib}

\begin{thebibliography}{10}

\bibitem{Akri2011}
G.~Akrivis, C.~Makridakis, and R.H. Nochetto.
\newblock Galerkin and {R}unge-{K}utta methods: unified formulation, a
  posteriori error estimates and nodal superconvergence.
\newblock {\em Numer. Math.}, 118(3):429--456, 2011.

\bibitem{Alexander1977}
R.~Alexander.
\newblock Diagonally implicit {R}unge-{K}utta methods for stiff o.d.e.'s.
\newblock {\em SIAM J. Numer. Anal.}, 14(6):1006--1021, 1977.

\bibitem{Elliott2021}
A.~Alphonse, D.~Caetano, A~Djurdjevac, and C.M. Elliott.
\newblock Function spaces, time derivatives and compactness for evolving
  families of banach spaces with applications to pdes, 2021.

\bibitem{Andrei2011}
B.~Andreianov.
\newblock Time compactness tools for discretized evolution equations and
  applications to degenerate parabolic {PDE}s.
\newblock In {\em Finite volumes for complex applications {VI}. {P}roblems \&
  perspectives. {V}olume 1, 2}, volume~4 of {\em Springer Proc. Math.}, pages
  21--29. Springer, Heidelberg, 2011.

\bibitem{Ascher1997}
U.M. Ascher, S.J. Ruuth, and R.J. Spiteri.
\newblock Implicit-explicit {R}unge-{K}utta methods for time-dependent partial
  differential equations.
\newblock {\em Appl. Numer. Math.}, 25(2-3):151--167, 1997.
\newblock Special issue on time integration (Amsterdam, 1996).

\bibitem{Aubin1963}
J.-P. Aubin.
\newblock Un th\'{e}or\`eme de compacit\'{e}.
\newblock {\em C. R. Acad. Sci. Paris}, 256:5042--5044, 1963.

\bibitem{Burrage1979}
K.~Burrage and J.C. Butcher.
\newblock Stability criteria for implicit {R}unge-{K}utta methods.
\newblock {\em SIAM J. Numer. Anal.}, 16(1):46--57, 1979.

\bibitem{Burrage1980}
K.~Burrage and J.C. Butcher.
\newblock Nonlinear stability of a general class of differential equation
  methods.
\newblock {\em BIT}, 20(2):185--203, 1980.

\bibitem{Butcher1975}
J.C. Butcher.
\newblock A stability property of implicit runge-kutta methods.
\newblock {\em BIT Numerical Mathematics}, 15(4):358--361, 1975.

\bibitem{Butcher2008}
J.C. Butcher.
\newblock {\em Numerical methods for ordinary differential equations}.
\newblock John Wiley \& Sons, Ltd., Chichester, second edition, 2008.

\bibitem{Jungel2014}
X.~Chen, A.~J\"{u}ngel, and J.-G. Liu.
\newblock A note on {A}ubin-{L}ions-{D}ubinski\u{\i} lemmas.
\newblock {\em Acta Appl. Math.}, 133:33--43, 2014.

\bibitem{Crockatt2019}
M.M. Crockatt, A.J. Christlieb, C.K. Garrett, and C.D. Hauck.
\newblock Hybrid methods for radiation transport using diagonally implicit
  {R}unge-{K}utta and space-time discontinuous {G}alerkin time integration.
\newblock {\em J. Comput. Phys.}, 376:455--477, 2019.

\bibitem{Crou1979}
M.~Crouzeix.
\newblock Sur la {$B$}-stabilit\'{e} des m\'{e}thodes de {R}unge-{K}utta.
\newblock {\em Numer. Math.}, 32(1):75--82, 1979.

\bibitem{Crouzeix1976}
M.~Crouzeix and P.-A. Raviart.
\newblock Approximation des \'{e}quations d'\'{e}volution lin\'{e}aires par des
  m\'{e}thodes \`a pas multiples.
\newblock {\em C. R. Acad. Sci. Paris S\'{e}r. A-B}, 28(6):Aiv, A367--A370,
  1976.

\bibitem{Dahlquist1978}
G.~Dahlquist.
\newblock {$G$}-stability is equivalent to {$A$}-stability.
\newblock {\em BIT}, 18(4):384--401, 1978.

\bibitem{Dahlquist1963}
G.G. Dahlquist.
\newblock A special stability problem for linear multistep methods.
\newblock {\em Nordisk Tidskr. Informationsbehandling (BIT)}, 3:27--43, 1963.

\bibitem{Jungel2012}
M.~Dreher and A.~J\"{u}ngel.
\newblock Compact families of piecewise constant functions in {$L^p(0,T;B)$}.
\newblock {\em Nonlinear Anal.}, 75(6):3072--3077, 2012.

\bibitem{Egger2019}
H.~Egger.
\newblock Structure preserving approximation of dissipative evolution problems.
\newblock {\em Numer. Math.}, 143(1):85--106, 2019.

\bibitem{Ehle1969}
B.L. Ehle.
\newblock {\em On Pad{\'e} approximations to the exponential function and
  A-stable methods for the numerical solution of initial value problems}.
\newblock PhD thesis, University of Waterloo Waterloo, Ontario, 1969.

\bibitem{Ehle1973}
B.L. Ehle.
\newblock {$A$}-stable methods and {P}ad\'{e} approximations to the
  exponential.
\newblock {\em SIAM J. Math. Anal.}, 4:671--680, 1973.

\bibitem{Eldred2020}
C.~Eldred and F.~Gay-Balmaz.
\newblock Single and double generator bracket formulations of multicomponent
  fluids with irreversible processes.
\newblock {\em J. Phys. A}, 53(39):395701, 34, 2020.

\bibitem{Ern2019}
A.~Ern, I.~Smears, and M.~Vohral\'{\i}k.
\newblock Equilibrated flux {\it a posteriori} error estimates in
  {$L^2(H^1)$}-norms for high-order discretizations of parabolic problems.
\newblock {\em IMA J. Numer. Anal.}, 39(3):1158--1179, 2019.

\bibitem{Evans1987}
L.C. Evans and R.F. Gariepy.
\newblock Blowup, compactness and partial regularity in the calculus of
  variations.
\newblock {\em Indiana Univ. Math. J.}, 36(2):361--371, 1987.

\bibitem{Feng2006}
X.~Feng.
\newblock Fully discrete finite element approximations of the
  {N}avier-{S}tokes-{C}ahn-{H}illiard diffuse interface model for two-phase
  fluid flows.
\newblock {\em SIAM J. Numer. Anal.}, 44(3):1049--1072, 2006.

\bibitem{Frank1985I}
R.~Frank, J.~Schneid, and C.W. Ueberhuber.
\newblock Order results for implicit {R}unge-{K}utta methods applied to stiff
  systems.
\newblock {\em SIAM J. Numer. Anal.}, 22(3):515--534, 1985.

\bibitem{Frank1985II}
R.~Frank, J.~Schneid, and C.W. Ueberhuber.
\newblock Stability properties of implicit {R}unge-{K}utta methods.
\newblock {\em SIAM J. Numer. Anal.}, 22(3):497--514, 1985.

\bibitem{Gllou2012}
T.~Gallou\"{e}t and J.-C. Latch\'{e}.
\newblock Compactness of discrete approximate solutions to parabolic
  {PDE}s---application to a turbulence model.
\newblock {\em Commun. Pure Appl. Anal.}, 11(6):2371--2391, 2012.

\bibitem{Gidas1981}
B.~Gidas and J.~Spruck.
\newblock A priori bounds for positive solutions of nonlinear elliptic
  equations.
\newblock {\em Comm. Partial Differential Equations}, 6(8):883--901, 1981.

\bibitem{GS2009}
J.-L. Guermond and A.~Salgado.
\newblock A splitting method for incompressible flows with variable density
  based on a pressure {P}oisson equation.
\newblock {\em J. Comput. Phys.}, 228(8):2834--2846, 2009.

\bibitem{GS2011}
J.-L. Guermond and A.J. Salgado.
\newblock Error analysis of a fractional time-stepping technique for
  incompressible flows with variable density.
\newblock {\em SIAM J. Numer. Anal.}, 49(3):917--944, 2011.

\bibitem{Hairer1993I}
E.~Hairer, S.P. N\o~rsett, and G.~Wanner.
\newblock {\em Solving ordinary differential equations. {I}}, volume~8 of {\em
  Springer Series in Computational Mathematics}.
\newblock Springer-Verlag, Berlin, second edition, 1993.
\newblock Nonstiff problems.

\bibitem{Hairer1996II}
E.~Hairer and G.~Wanner.
\newblock {\em Solving ordinary differential equations. {II}}, volume~14 of
  {\em Springer Series in Computational Mathematics}.
\newblock Springer-Verlag, Berlin, second edition, 1996.
\newblock Stiff and differential-algebraic problems.

\bibitem{Osterman2010}
E.~Hansen and A.~Ostermann.
\newblock Unconditional convergence of {DIRK} schemes applied to dissipative
  evolution equations.
\newblock {\em Appl. Numer. Math.}, 60(1-2):55--63, 2010.

\bibitem{Higueras2005}
I.~Higueras.
\newblock Monotonicity for {R}unge-{K}utta methods: inner product norms.
\newblock {\em J. Sci. Comput.}, 24(1):97--117, 2005.

\bibitem{Hochbruck2018}
M.~Hochbruck, T.~Pa\v{z}ur, and R.~Schnaubelt.
\newblock Error analysis of implicit {R}unge-{K}utta methods for quasilinear
  hyperbolic evolution equations.
\newblock {\em Numer. Math.}, 138(3):557--579, 2018.

\bibitem{Humph1994}
A.R. Humphries and A.M. Stuart.
\newblock Runge-{K}utta methods for dissipative and gradient dynamical systems.
\newblock {\em SIAM J. Numer. Anal.}, 31(5):1452--1485, 1994.

\bibitem{Carpenter2016}
C.A. Kennedy and M.H. Carpenter.
\newblock {\em Diagonally Implicit Runge–Kutta Methods for Ordinary
  Differential Equations: A Review}.
\newblock NASA Langley Research Center, 2016.

\bibitem{Kraaije1989}
J.F.B.~M. Kraaijevanger and M.N. Spijker.
\newblock Algebraic stability and error propagation in {R}unge-{K}utta methods.
\newblock volume~5, pages 71--87. 1989.
\newblock Recent theoretical results in numerical ordinary differential
  equations.

\bibitem{Layton2008}
W.~Layton.
\newblock {\em Introduction to the numerical analysis of incompressible viscous
  flows}, volume~6 of {\em Computational Science \& Engineering}.
\newblock Society for Industrial and Applied Mathematics (SIAM), Philadelphia,
  PA, 2008.
\newblock With a foreword by Max Gunzburger.

\bibitem{Lions1969}
J.-L. Lions.
\newblock {\em Quelques m\'{e}thodes de r\'{e}solution des probl\`emes aux
  limites non lin\'{e}aires}.
\newblock Dunod; Gauthier-Villars, Paris, 1969.

\bibitem{Lubich1995}
C.~Lubich and A.~Ostermann.
\newblock Runge-{K}utta approximation of quasi-linear parabolic equations.
\newblock {\em Math. Comp.}, 64(210):601--627, 1995.

\bibitem{Makri2006}
C.~Makridakis and R.H. Nochetto.
\newblock A posteriori error analysis for higher order dissipative methods for
  evolution problems.
\newblock {\em Numer. Math.}, 104(4):489--514, 2006.

\bibitem{Necas1996}
J.~Malek, J.~Necas, M.~Rokyta, and M.~Ruzicka.
\newblock {\em Weak and measure-valued solutions to evolutionary {PDE}s},
  volume~13 of {\em Applied Mathematics and Mathematical Computation}.
\newblock Chapman \& Hall, London, 1996.

\bibitem{MarionTemam1998}
M.~Marion and R.~Temam.
\newblock Navier-{S}tokes equations: theory and approximation.
\newblock In {\em Handbook of numerical analysis, {V}ol. {VI}}, Handb. Numer.
  Anal., VI, pages 503--688. North-Holland, Amsterdam, 1998.

\bibitem{Mielke2011}
A.~Mielke.
\newblock Formulation of thermoelastic dissipative material behavior using
  {GENERIC}.
\newblock {\em Contin. Mech. Thermodyn.}, 23(3):233--256, 2011.

\bibitem{Ottinger1997}
H.C. \"{O}ttinger and M.~Grmela.
\newblock Dynamics and thermodynamics of complex fluids. {II}. {I}llustrations
  of a general formalism.
\newblock {\em Phys. Rev. E (3)}, 56(6):6633--6655, 1997.

\bibitem{Pazner2017}
W.~Pazner and P.-O. Persson.
\newblock Stage-parallel fully implicit {R}unge-{K}utta solvers for
  discontinuous {G}alerkin fluid simulations.
\newblock {\em J. Comput. Phys.}, 335:700--717, 2017.

\bibitem{Ranocha2020}
H.~Ranocha and D.I. Ketcheson.
\newblock Energy stability of explicit {R}unge-{K}utta methods for
  nonautonomous or nonlinear problems.
\newblock {\em SIAM J. Numer. Anal.}, 58(6):3382--3405, 2020.

\bibitem{Ranocha2018}
H.~Ranocha and P.~\"{O}ffner.
\newblock {$L_2$} stability of explicit {R}unge-{K}utta schemes.
\newblock {\em J. Sci. Comput.}, 75(2):1040--1056, 2018.

\bibitem{Roubicek2013}
T.~Roub\'{\i}\v{c}ek.
\newblock {\em Nonlinear partial differential equations with applications},
  volume 153 of {\em International Series of Numerical Mathematics}.
\newblock Birkh\"{a}user/Springer Basel AG, Basel, second edition, 2013.

\bibitem{Shen2010}
J.~Shen and X.~Yang.
\newblock Numerical approximations of {A}llen-{C}ahn and {C}ahn-{H}illiard
  equations.
\newblock {\em Discrete Contin. Dyn. Syst.}, 28(4):1669--1691, 2010.

\bibitem{Shin2020}
J.~Shin and J.-Y. Lee.
\newblock An energy stable {R}unge-{K}utta method for convex gradient problems.
\newblock {\em J. Comput. Appl. Math.}, 367:112455, 12, 2020.

\bibitem{Simon1987}
J.~Simon.
\newblock Compact sets in the space {$L^p(0,T;B)$}.
\newblock {\em Ann. Mat. Pura Appl. (4)}, 146:65--96, 1987.

\bibitem{Smears2017}
I.~Smears.
\newblock Robust and efficient preconditioners for the discontinuous {G}alerkin
  time-stepping method.
\newblock {\em IMA J. Numer. Anal.}, 37(4):1961--1985, 2017.

\bibitem{HumphBook}
A.M. Stuart and A.R. Humphries.
\newblock {\em Dynamical systems and numerical analysis}, volume~2 of {\em
  Cambridge Monographs on Applied and Computational Mathematics}.
\newblock Cambridge University Press, Cambridge, 1996.

\bibitem{Wu2022}
Z.~Sun, Y.~Wei, and K.~Wu.
\newblock On energy laws and stability of runge--kutta methods for linear
  seminegative problems, 2022.

\bibitem{Tartar1979}
L.~Tartar.
\newblock Compensated compactness and applications to partial differential
  equations.
\newblock In {\em Nonlinear analysis and mechanics: {H}eriot-{W}att
  {S}ymposium, {V}ol. {IV}}, volume~39 of {\em Res. Notes in Math.}, pages
  136--212. Pitman, Boston, Mass.-London, 1979.

\bibitem{Temam1979}
R.~Temam.
\newblock {\em Navier-{S}tokes equations}, volume~2 of {\em Studies in
  Mathematics and its Applications}.
\newblock North-Holland Publishing Co., Amsterdam-New York, revised edition,
  1979.
\newblock Theory and numerical analysis, With an appendix by F. Thomasset.

\bibitem{Thomee2006}
V.~Thom\'{e}e.
\newblock {\em Galerkin finite element methods for parabolic problems},
  volume~25 of {\em Springer Series in Computational Mathematics}.
\newblock Springer-Verlag, Berlin, second edition, 2006.

\bibitem{Walk2010}
N.J. Walkington.
\newblock Compactness properties of the {DG} and {CG} time stepping schemes for
  parabolic equations.
\newblock {\em SIAM J. Numer. Anal.}, 47(6):4680--4710, 2010.

\bibitem{Walk2014}
N.J. Walkington.
\newblock Combined {DG}-{CG} time stepping for wave equations.
\newblock {\em SIAM J. Numer. Anal.}, 52(3):1398--1417, 2014.

\end{thebibliography}


\end{document}